\documentclass{article}

\usepackage[australian]{babel}
\usepackage{fullpage}

\usepackage{hyperref}

\usepackage{amssymb,amsmath,amsfonts}
\usepackage{amsthm}
\usepackage[all]{xy}
\usepackage[capitalise]{cleveref}

\allowdisplaybreaks
\numberwithin{equation}{section}

\usepackage{xcolor}
\definecolor{MyBlue}{cmyk}{1,0.13,0,0.63}
\definecolor{MyGreen}{cmyk}{0.91,0,0.88,0.52}
\newcommand{\mylinkcolor}{MyBlue}
\newcommand{\mycitecolor}{MyGreen}

\makeatletter
\def\@endtheorem{\endtrivlist}
\makeatother
\theoremstyle{plain}
\newtheorem{thm}{Theorem}[section]

\newtheorem*{main*}{Main Theorem}
\newtheorem{lem}[thm]{Lemma}
\newtheorem{prop}[thm]{Proposition}
\newtheorem{coro}[thm]{Corollary}

\theoremstyle{definition}
\newtheorem{defn}[thm]{Definition}
\newtheorem{remark}[thm]{Remark}
\newtheorem{assumption}[thm]{Assumption}
\newtheorem*{standing}{Standing Assumptions}
\newtheorem{example}[thm]{Example}

\makeatletter
\newcommand{\customlabel}[2]{%
   \protected@write \@auxout {}{\string \newlabel {#1}{{#2}{\thepage}{#2}{#1}{}} }%
   \hypertarget{#1}{}
}
\makeatother

\renewcommand{\eqref}[1]{\labelcref{#1}}
\crefname{thm}{Theorem}{Theorems}
\crefname{lem}{Lemma}{Lemmas}
\crefname{prop}{Proposition}{Propositions}
\crefname{coro}{Corollary}{Corollaries}
\crefname{defn}{Definition}{Definitions}
\crefname{example}{Example}{Examples}
\crefname{remark}{Remark}{Remarks}

\hypersetup{%
  bookmarksnumbered=true,bookmarksopen=false,%
  plainpages=false,
  linktocpage=true,%
  colorlinks=true,breaklinks=true,%
  linkcolor=\mylinkcolor,citecolor=\mycitecolor,%
  pdfpagelayout=OneColumn,%
  pageanchor=true,%
}

\makeatletter
\def\thm@space@setup{%
  \thm@preskip=4pt plus 2pt minus 2pt
  \thm@postskip=\thm@preskip
}
\renewenvironment{proof}[1][\proofname]{\par
  \pushQED{\qed}%
  \normalfont \topsep4\p@\relax 
  \trivlist
  \item[\hskip\labelsep
        \itshape
    #1\@addpunct{.}]\ignorespaces
}{%
  \popQED\endtrivlist\@endpefalse
}
\makeatother

\usepackage{enumitem}
\setlist{topsep=4pt plus 2pt minus 2pt,partopsep=0pt,itemsep=2pt plus 2pt minus 2pt,parsep=0.5\parskip}
\setenumerate[1]{label=\arabic*)}

\newcommand{\MR}[1]{}

\let\OLDthebibliography\thebibliography
\renewcommand\thebibliography[1]{
  \addcontentsline{toc}{section}{\refname}
  \OLDthebibliography{CPRS06b}
  \setlength{\parskip}{0pt}
  \setlength{\itemsep}{0pt plus 0.3ex}
}

\newcommand{\N}{\mathbb{N}}
\newcommand{\R}{\mathbb{R}}
\newcommand{\C}{\mathbb{C}}
\newcommand{\Z}{\mathbb{Z}}
\newcommand{\A}{\mathcal{A}}
\newcommand{\mH}{\mathcal{H}}
\newcommand{\mK}{\mathcal{K}}
\newcommand{\D}{\mathcal{D}}
\newcommand{\B}{\mathcal{B}}
\newcommand{\E}{\mathcal{E}}
\newcommand{\pS}{\mathcal{S}}

\DeclareMathOperator{\Dom}{Dom}
\DeclareMathOperator{\Ran}{Ran}

\DeclareMathOperator{\End}{End}
\DeclareMathOperator{\FEnd}{FEnd}
\DeclareMathOperator{\Hom}{Hom}
\DeclareMathOperator{\supp}{supp}
\DeclareMathOperator{\Index}{Index}
\DeclareMathOperator{\ev}{ev}
\DeclareMathOperator{\id}{id}

\renewcommand{\bar}[1]{\overline{#1}}

\newcommand{\CCliff}{{\mathbb{C}\mathrm{l}}}

\newcommand{\dvol}{\textnormal{dvol}}
\newcommand{\op}{\textnormal{op}}
\newcommand{\K}{K}
\newcommand{\KK}{K\!K}
\newcommand{\SF}{\textnormal{sf}}
\newcommand{\ASF}{\textnormal{SF}}
\newcommand{\cyl}{\textnormal{cyl}}
\newcommand{\til}[1]{\widetilde{#1}}
\newcommand{\hotimes}{\mathbin{\hat\otimes}}
\newcommand{\hot}{\hotimes}

\newcommand{\mvert}{\,|\,}
\newcommand{\bigmvert}{\,\big|\,}
\newcommand{\la}{\langle}
\newcommand{\ra}{\rangle}
\newcommand{\into}{\hookrightarrow}

\newcommand{\bF}{F}

\newcommand{\mattwo}[4]{
  \left(\!\!\!\begin{array}{c@{~}c}#1&#2\\ #3&#4\\\end{array}\!\!\!\right)
}
\newcommand{\vectwo}[2]{
  \left(\!\!\!\begin{array}{c}#1\\ #2\\\end{array}\!\!\!\right)
}
\newcommand{\matfour}[4]{
  \left(\!\!\begin{array}{c@{~}c@{~}c@{~}c}#1\\ #2\\ #3\\ #4\\\end{array}\!\!\right)
}

\title{The index of generalised Dirac-Schr\"odinger operators}
\author{
Koen van den Dungen%
\footnote{\emph{Present address:} Mathematisches Institut der Universit\"at Bonn, Endenicher Allee 60, D-53115 Bonn, \texttt{kdungen@uni-bonn.de}}\\[4mm]
{\normalsize SISSA (Scuola Internazionale Superiore di Studi Avanzati)}\\ 
{\normalsize Via Bonomea, 265, 34136 Trieste, Italy}
}
\date{}

\begin{document}

\maketitle

\begin{abstract}
\noindent
We study the relation between spectral flow and index theory within the framework of (unbounded) $\KK$-theory. In particular, we consider a generalised notion of `Dirac-Schr\"odinger operators', consisting of a self-adjoint elliptic first-order differential operator $\D$ with a skew-adjoint `potential' given by a (suitable) family of unbounded operators on an auxiliary Hilbert module. 
We show that such Dirac-Schr\"odinger operators are Fredholm, and we prove a relative index theorem for these operators (which allows cutting and pasting of the underlying manifolds). 
Furthermore, we show that the index of a Dirac-Schr\"odinger operator represents the pairing (Kasparov product) of the $\K$-theory class of the potential with the $\K$-homology class of $\D$. 
We prove this result without assuming that the potential is differentiable; instead, we assume that the `variation' of the potential is sufficiently small near infinity. 
In the special case of the real line, we recover the well-known equality of the index with the spectral flow of the potential. 

\vspace{\baselineskip}

\end{abstract}

\section{Introduction}

A `classical' \emph{Dirac-Schr\"odinger operator} (also called a \emph{Callias-type} operator) is an operator of the form $\D-iV$, where $\D$ is a Dirac-type operator, and the `potential' $V$ is a self-adjoint endomorphism on some auxiliary vector bundle (of finite rank). Under suitable assumptions on the potential $V$, one can then prove that $\D-iV$ is Fredholm (see, for instance, \cite{Cal78,Ang90,BM92,Ang93a,Rad94,Bun95}). 

In this paper we aim to prove a similar statement in the case where the auxiliary vector bundle is of \emph{infinite} rank, and the `potential' consists of a family of unbounded operators. 
Since any locally trivial Hilbert bundle with (separable) infinite-dimensional fibres is in fact globally trivial (see \cite[Theorem 10.8.8]{Dixmier82}), we will restrict our attention to the case of a globally trivial Hilbert bundle with infinite-dimensional fibre $\mH$. 
(In fact, instead of $\mH$ we will more generally consider a Hilbert module over an arbitrary $\sigma$-unital $C^*$-algebra, but in this introduction we limit our attention to the simpler case of a Hilbert space.)

One motivation for studying the case of a potential acting on infinite-dimen\-sional fibres comes from the notion of spectral flow. Namely, there is a well-known `index = spectral flow' equality, which states that the spectral flow of a (suitably continuous) family of unbounded self-adjoint operators $\{\pS(x)\}_{x\in[0,1]}$ (with invertible endpoints) is equal to the index of $\partial_x+\pS(\cdot)$ (see e.g.\ \cite{RS95,Wah07,AW11}). Here, we note that the operator $-i\partial_x-i\pS(\cdot)$ is a special case of a Dirac-Schr\"odinger operator, where the underlying manifold is just the real line. The operator $\pS(\cdot)$ is a regular self-adjoint Fredholm operator on $C_0(\R,\mH)$, which defines a class $[\pS(\cdot)]$ in the odd $\K$-theory group $\KK^1(\C,C_0(\R)) \simeq \K_1(C_0(\R))$. Under the Bott periodicity isomorphism $\beta\colon\K_1(C_0(\R))\xrightarrow{\simeq}\K_0(\C)=\Z$, the class $[\pS(\cdot)]$ corresponds precisely to the spectral flow of $\{\pS(x)\}_{x\in[0,1]}$. By observing that the Bott periodicity isomorphism is implemented as the Kasparov product with the generator $[-i\partial_x]$ of the $\K$-homology group $\KK^1(C_0(\R),\C)$, we see that we have the equality
\[
\Index\big( \partial_x+\pS(\cdot) \big) = [\pS(\cdot)] \otimes_{C_0(\R)} [-i\partial_x] .
\]
Our main goal in this paper is to generalise this equality, replacing $\R$ by an arbitrary manifold $M$, and $-i\partial_x$ by some first-order differential operator on $M$. 
Such a generalisation has already been obtained by Kaad and Lesch \cite[\S8]{KL13}, under the assumption that the family of operators $\{\pS(x)\}_{x\in M}$ is \emph{differentiable} (in a suitable sense). However, on the real line, such a differentiability assumption is not necessary to obtain the above-mentioned `index = spectral flow' equality \cite[Theorem 2.1]{AW11}. 
Our aim is therefore to make a link between these two approaches, by studying Dirac-Schr\"odinger operators on higher-dimensional manifolds (as in \cite{KL13}), without assuming that the family of operators is differentiable (as in \cite{Wah07,AW11}). 

More precisely, we will consider the following setup. 
Let $M$ be a connected Riemannian manifold, and let $\D$ be a self-adjoint elliptic first-order differential operator on $M$.
Let $\{\pS(x)\}_{x\in M}$ be a family of self-adjoint operators with compact resolvents and with common domain $W$ on a Hilbert space $\mH$ such that $\pS\colon M\to\B(W,\mH)$ is norm-continuous, 
and $\pS(x)$ is uniformly invertible outside a compact subset $K\subset M$. 
We then consider the `generalised Dirac-Schr\"odinger operator' $\D_\pS := \D - i \pS(\cdot)$. 

The family $\{\pS(x)\}_{x\in M}$ on $\mH$ defines a regular self-adjoint Fredholm operator $\pS(\cdot)$ on the Hilbert $C_0(M)$-module $C_0(M,\mH)$, which defines a class $[\pS(\cdot)]$ in the odd $\K$-theory group $\KK^1(\C,C_0(M)) \simeq \K_1(C_0(M))$. 
The operator $\D$ yields an odd spectral triple and hence a $\K$-homology class $[\D] \in \KK^1(C_0(M),\C)$. 
The pairing between $[\pS(\cdot)]$ and $[\D]$, given by the Kasparov product, then yields a class in $\KK^0(\C,\C) \simeq \Z$. 
Our main result states that this pairing can be computed as the index of $\D - i \pS(\cdot)$ (see \cref{thm:sum_Fredholm,thm:Kasp_prod_index}), assuming that the `variation' of the family $\{\pS(x)\}_{x\in M}$ is `sufficiently small' outside of $K$, in the following precise sense. 

\begin{main*}
Suppose that there exists a disjoint finite open cover $\{V_j\}$ of $M\backslash K$ with points $x_j\in\bar V_j$ and positive numbers $a_j<1$ such that $\big\| \big( \pS(x)-\pS(x_j) \big) \pS(x_j)^{-1} \big\| \leq a_j$ for all $x\in V_j$. 
Then the Dirac-Schr\"odinger operator $\D_\pS := \D-i\pS(\cdot)$ on $L^2(M,\mH\otimes\bF)$ is Fredholm, and its index computes the pairing (Kasparov product) of the odd $\K$-theory class $[\pS(\cdot)] \in \KK^1(\C,C_0(M))$ with the odd $\K$-homology class $[\D] \in \KK^1(C_0(M),\C)$: 
\[
\Index\left( \D-i\pS(\cdot) \right) = [\pS(\cdot)] \otimes_{C_0(M)} [\D] \in \KK^0(\C,\C) \simeq \Z .
\]
\end{main*}

Our main theorem is complementary to the results of \cite[\S8]{KL13}; although we do not need to assume any differentiability of $\pS(\cdot)$, we do need stronger assumptions on the `variation' of $\pS(\cdot)$ near infinity. 
Furthermore, our theorem generalises the aforementioned `index = spectral flow' equality of \cite{AW11} (which is obtained in the special case $M=\R$). 

The index of $\D-i\pS(\cdot)$ corresponds to the Kasparov class of the self-adjoint Fredholm operator
\begin{align}
\label{eq:prod_op}
\til\D_\pS := \pS(\cdot)\times\D := \mattwo{0}{\D+i\pS(\cdot)}{\D-i\pS(\cdot)}{0} .
\end{align}
This \emph{product operator} is precisely given by the standard formula for the construction of the unbounded Kasparov product \cite{Mes14,KL13,BMS16,MR16}. 
Unfortunately, we cannot directly apply these results on the construction of the unbounded Kasparov product, since we do not assume any differentiability for the potential $\pS(\cdot)$. 
However, it is interesting to note that in our scenario, despite this lack of differentiability, 
the formula \eqref{eq:prod_op} for the unbounded Kasparov product nevertheless remains correct. 

Let us provide a brief outline of this paper. In \cref{sec:KK}, we start by recalling the main facts regarding (unbounded) $\KK$-theory and the Kasparov product. Next, we discuss in detail how an unbounded Fredholm operator defines a class in $\K$-theory, following the approach of \cite{Wah07}. We also provide a purely $\KK$-theoretic description of the notion of spectral flow. 

In \cref{sec:DS} we define our notion of generalised Dirac-Schr\"odinger operators, and we prove that they are self-adjoint. In \cref{sec:index}, we show that these Dirac-Schr\"odinger operators are also Fredholm. Furthermore, we prove a relative index theorem, generalising Bunke's $\K$-theoretic relative index theorem \cite{Bun95}. 

In \cref{sec:Kasp_prod} we prove our main theorem, stating that the index of the Dirac-Schr\"odinger operator $\D_\pS$ is equal to the Kasparov product of $[\pS(\cdot)]$ with $[\D]$. 
First, using the same methods as in \cite{KL13}, we can prove the theorem for a \emph{differentiable} family. 
The general case is then obtained by showing that we can always find a homotopy which replaces our continuous family $\{\pS(x)\}_{x\in M}$ by a differentiable family.
Here we make use of a consequence of the relative index theorem, which shows that the index of a Dirac-Schr\"odinger operator is unaffected, if one replaces a suitable open subset of the manifold by a cylindrical end.

\subsection*{Acknowledgements}

This paper was partially supported by the grant H2020-MSCA-RISE-2015-691246-QUANTUM DYNAMICS, which funded the author's visit to the University of Wollongong in February 2017. 
The author thanks Adam Rennie for his hospitality during this visit, and for several helpful discussions. 
The author also thanks Matthias Lesch for interesting discussions and helpful comments, and for his hospitality during a short visit to the University of Bonn (late June -- early July 2017), which was funded by the COST Action MP1405 QSPACE, supported by COST (European
Cooperation in Science and Technology). 
Finally, thanks to the referee for helpful comments and suggestions.

\section{Index and spectral flow in unbounded KK-theory}
\label{sec:KK}

We start this section by recalling the main definitions and facts regarding Hilbert modules. For a more detailed introduction, we refer to \cite{Blackadar98,Lance95}. 
Let $B$ be a (trivially graded) $\sigma$-unital $C^*$-algebra. Recall that a $\Z_2$-graded Hilbert $B$-module is a vector space $E = E^0\oplus E^1$ equipped with a right action $E\times B\to E$ (satisfying $E^j\times B\to E^j$ for $j=0,1$) and with a $B$-valued inner product $\la\cdot|\cdot\ra\colon E\times E\to B$ (satisfying $\la E^0|E^1\ra=\{0\}$), such that $E$ is complete in the corresponding norm. 
The endomorphisms $\End_B(E)$ are the adjointable linear operators $E\to E$.
For any $\psi,\eta\in E$, the rank-one operators $\theta_{\psi,\eta}$ are defined by $\theta_{\psi,\eta}\xi := \psi \la\eta|\xi\ra$ for $\xi\in E$. 
The compact endomorphisms $\End^0_B(E)$ are given by the closure of the space of finite linear combinations of rank-one operators. 
For two Hilbert $B$-modules $E_1$ and $E_2$, the adjointable linear operators $E_1\to E_2$ are denoted by $\Hom_B(E_1,E_2)$. To avoid confusion, we sometimes denote the norm on $\Hom_B(E_1,E_2)$ as $\|\cdot\|_{E_1\to E_2}$. 

A Hilbert module is called \emph{countably generated} if there is a countable set $\{\psi_n\}\subset E$ such that the linear combinations of elements $\psi_n b$ (with $b\in B$) form a dense subset of $E$. 
The \emph{standard Hilbert module} over $B$ is defined as 
\[
\mH_B := l^2(\N\to B) := \Big\{ \{b_n\}_{n\in\N} : \sum_{n\in\N} b_n^* b_n \text{ is norm-convergent in } B \Big\} , 
\]
equipped with the inner product $\big\la \{a_n\} \bigmvert \{b_n\} \big\ra := \sum_{n\in\N} a_n^* b_n$. Since $B$ is $\sigma$-unital, the standard Hilbert module $\mH_B$ is countably generated. By Kasparov's stabilisation theorem \cite[Theorem 13.6.2]{Blackadar98}, we know that every countably generated Hilbert $B$-module $E$ appears as a direct summand of the standard Hilbert module: to be precise we have $E\oplus\mH_B \simeq \mH_B$. 

Now consider two ($\sigma$-unital) $C^*$-algebras $A$ and $B$. Let $E$ be a Hilbert $A$-module, and let $F$ be a Hilbert $B$-module, equipped with a $*$-homomorphism $\pi\colon A\to\End_B(F)$. 
The algebraic tensor product $E\odot_AF$ is given by finite sums of simple tensors such that $ea\otimes f = e\otimes\pi(a)f$ for all $e\in E$, $f\in F$, and $a\in A$. 
We then construct the \emph{interior tensor product} $E\otimes_AF$ (also called the \emph{balanced tensor product}) as the completion of $E\odot_AF$ with respect to the inner product $\la e_1\otimes f_1 | e_2\otimes f_2 \ra := \big\la f_1 \bigmvert \pi(\la e_1|e_2\ra) f_2 \big\ra$. 

A densely defined operator $S$ on $E$ is called \emph{semi-regular} if the adjoint $S^*$ is densely defined. 
A semi-regular operator $S$ is called \emph{regular} if $S$ is closed and $1+S^*S$ has dense range. 
A densely defined, closed, symmetric operator $S$ is regular and self-adjoint if and only if the operators $S\pm i$ are surjective \cite[Lemma 9.8]{Lance95}. 
We say that a semi-regular operator $S$ is \emph{essentially} regular self-adjoint if its closure $\bar S$ is regular self-adjoint. 
If $B=\C$, then a Hilbert $\C$-module is just a Hilbert space $\mH$, and we write $\B(\mH) = \End_\C(\mH)$ and $\mK(\mH) = \End_\C^0(\mH)$. In this case, any closed operator on $\mH$ is regular. 
Regular operators on a Hilbert $B$-module have similar properties as closed operators on Hilbert spaces. In particular, there is a continuous functional calculus for regular self-adjoint operators \cite{Wor91,Kus97pre,Kuc02}. 

Given a densely defined, symmetric operator $S$ on $E$, we can equip $\Dom S$ with the graph inner product $\la\psi|\psi\ra_S = \la(S\pm i)\psi|(S\pm i)\psi\ra = \la\psi|\psi\ra + \la S\psi|S\psi\ra$. The graph norm of $S$ is then defined as $\|\psi\|_S := \big\|\la\psi|\psi\ra_S\big\|^{\frac12}$. 

We prove a few basic lemmas which will be used in this paper. 
\begin{lem}
\label{lem:rel_bdd}
Let $T$ be a regular self-adjoint operator on a Hilbert $B$-module $E$, and let $S$ be an operator such that $\Dom S\supset\Dom T$. For some $a\geq0$, we have $\|S(T\pm i)^{-1}\| \leq a$ if and only if $\|S\psi\| \leq a \|\psi\|_T$ for any $\psi\in\Dom T$. 
\end{lem}
\begin{proof}
We note that $\|\psi\|_T = \|(T\pm i)\psi\|$. 
If $\|S(T\pm i)^{-1}\| \leq a$ we calculate
\begin{align*}
\|S\psi\| &= \|S(T\pm i)^{-1}(T\pm i)\psi\| \leq a \|(T\pm i)\psi\| = a \|\psi\|_T .
\end{align*}
Conversely, if $\|S\psi\| \leq a \|\psi\|_T$ we have
\begin{align*}
\|S(T\pm i)^{-1}\| &= \sup_{\|\psi\|=1} \|S(T\pm i)^{-1}\psi\| \leq a \cdot \sup_{\|\psi\|=1} \|(T\pm i)^{-1}\psi\|_T = a .
\qedhere
\end{align*}
\end{proof}

The following lemma is a consequence of the closed graph theorem. 
\begin{lem}[{\cite[Lemma 6.5]{vdD16toappear}}]
\label{lem:dom_inc_rel_bdd}
Let $S$ be a closed operator on a Hilbert $B$-module $E$, and let $T$ be a closable operator such that $\Dom S\subset\Dom T$. Then $T$ is relatively bounded by $S$. 
\end{lem}

\begin{lem}
\label{lem:dom_sum}
Let $S$ and $T$ be closed symmetric operators on a Hilbert $B$-module $E$, such that $\Dom S\cap\Dom T$ is dense in $E$. Let $\Dom(S+T)$ denote the closure of $\Dom S\cap\Dom T$ in the graph norm of $S+T$. If $\Dom(S+T) \subset \Dom S$, then also $\Dom(S+T)\subset\Dom T$, and therefore $\Dom(S+T)\subset\Dom S\cap\Dom T$. 
\end{lem}
\begin{proof}
Since $\Dom(S+T) \subset \Dom S$, we know 
from \cref{lem:dom_inc_rel_bdd}
that there exists $C>0$ such that $\|S\psi\| \leq C \|\psi\|_{S+T}$ for all $\psi\in\Dom(S+T)$. Let $\psi\in\Dom(S+T)$, and consider a sequence $\psi_n\in\Dom S\cap\Dom T$ 
which converges to $\psi$ in the graph norm of $S+T$. 
Then we also have $S\psi_n\to S\psi$, because $\|S(\psi_n-\psi)\| \leq C \|\psi_n-\psi\|_{S+T}$. Hence $T\psi_n = (S+T)\psi_n - S\psi_n$ also converges, which means that $\psi\in\Dom T$. 
\end{proof}

Finally, we mention the Kato-Rellich theorem for regular self-adjoint operators on Hilbert modules.
\begin{thm}[{\cite[Theorem 4.5]{KL12}}]
Let $T$ be a regular self-adjoint operator on a Hilbert $B$-module $E$, and let $S$ be a symmetric operator such that $\Dom S\supset\Dom T$. Suppose there exist $a\in(0,1)$ and $b\in[0,\infty)$ such that $\|S\psi\| \leq a \|T\psi\| + b \|\psi\|$ for any $\psi\in\Dom T$. 
Then $T+S$ is regular and self-adjoint on the domain $\Dom(T+S) = \Dom T$. 
\end{thm}
We remark in particular that, if $\|S(T\pm i)^{-1}\| < 1$, then we obtain the inequality $\|S\psi\| \leq a \|T\psi\| + b \|\psi\|$ with $a=b=\|S(T\pm i)^{-1}\|$, so the Kato-Rellich theorem applies.

\subsection{Kasparov modules}

We consider two $C^*$-algebras $A$ and $B$. 
In this article, almost all $C^*$-algebras will be trivially graded. The only exception is an algebra of the form $A\otimes\CCliff_1$, where $A$ is trivially graded and $\CCliff_1$ is the Clifford algebra with one odd generator. 

Kasparov \cite{Kas80b} defined the abelian group $\KK(A,B)$ as a set of homotopy equivalence classes of certain Kasparov $A$-$B$-modules. 
Below we briefly recall the main definitions. 
More details can be found in e.g.\ \cite[\S17]{Blackadar98}. 

\begin{defn}
\label{defn:Kasp_mod}
An (even, bounded) \emph{Kasparov $A$-$B$-module} $(A,{}_{\pi}E_B,F)$ is given by a $\Z_2$-graded, countably generated, right Hilbert $B$-module $E$, a ($\Z_2$-graded) $*$-homomorphism $\pi\colon A\to\End_B(E)$, and an odd adjointable endomorphism $F\in\End_B(E)$ such that $\pi(a)(F-F^*)$, $[F,\pi(a)]_\pm$, and $\pi(a)(F^2-1)$ are compact endomorphisms. 

An \emph{odd} Kasparov $A$-$B$-module $(A,{}_{\pi}E_B,F)$ is defined in the same way as above, except that $A$, $B$, and $E$ are assumed to be trivially graded, and $F$ is not required to be odd. 

Two Kasparov modules $(A,{}_{\pi_1}{E_1}_B,F_1)$ and $(A,{}_{\pi_2}{E_2}_B,F_2)$ are called \emph{unitarily equivalent} if there exists an even unitary in $\Hom_B(E_1,E_2)$ intertwining the $\pi_j$ and $F_j$ (for $j=1,2$). 
A \emph{homotopy} between $(A,{}_{\pi_0}{E_0}_B,F_0)$ and $(A,{}_{\pi_1}{E_1}_B,F_1)$ is given by a Kasparov $A$-$C([0,1],B)$-module $(A,{}_{\til\pi}{\til E}_{C([0,1],B)},\til F)$ such that (for $j=0,1$)
\[
\ev_j(A,{}_{\til\pi}{\til E}_{C([0,1],B)},\til F) \simeq (A,{}_{\pi_j}{E_j}_B,F_j) .
\]
Here $\simeq$ denotes unitary equivalence, and $\ev_t(A,{}_{\til\pi}\til E,\til F) := (A,{}_{\til\pi\hot1}\til E\hot_{\rho_t}B,\til F\hot1)$, where the $*$-homomorphism $\rho_t\colon C([0,1],B) \to B$ is given by $\rho_t(b) := b(t)$. 

The (even) $\KK$-theory $\KK(A,B) = \KK^0(A,B)$ of $A$ and $B$ is defined as the set of homotopy equivalence classes of (even, bounded) Kasparov $A$-$B$-modules. 
Furthermore, the odd $\KK$-theory of $A$ and $B$ is defined as $\KK^1(A,B) := \KK(A\otimes\CCliff_1,B)$. 
\end{defn}

Since homotopy equivalence respects direct sums, the direct sum induces a (commutative and associative) binary operation (`addition') on the elements of $\KK(A,B)$. Given any Kasparov module $(A,{}_{\pi}E_B,F)$, let $E^\op:=E$ be equipped with the opposite grading (i.e.\ $(E^\op)^0 := E^1$ and $(E^\op)^1 := E^0$), and for $a=a^0+a^1\in A=A^0\oplus A^1$ define $\pi^\op(a) := \pi(a^0) - \pi(a^1)$. Then $(A,{}_{\pi^\op}E^\op_B,-F)$ is the additive inverse of $(A,{}_{\pi}E_B,F)$, i.e.\ the direct sum $(A,{}_{\pi}E_B,F) \oplus (A,{}_{\pi^\op}E^\op_B,-F)$ represents the trivial element in $\KK(A,B)$. Hence we find that $\KK(A,B)$ is in fact an abelian group \cite[Proposition 17.3.3]{Blackadar98}. 

In this article we will mostly focus on the unbounded representatives of $\KK$-elements introduced by Baaj and Julg \cite{BJ83}. 

\begin{defn}[\cite{BJ83}]
An (even) \emph{unbounded Kasparov $A$-$B$-module} $(\A,{}_{\pi}E_B,\D)$ is given by a $\Z_2$-graded, countably generated, right Hilbert $B$-module $E$, a ($\Z_2$-graded) $*$-homomorphism $\pi\colon A\to\End_B(E)$, a separable dense $*$-subalgebra $\A\subset A$, and a regular self-adjoint odd operator $\D\colon\Dom\D\subset E\to E$ such that 
\begin{enumerate}
\item we have the inclusion $\pi(\A)\cdot\Dom\D\subset\Dom\D$, and $[\D,\pi(a)]_\pm$ is bounded on $\Dom\D$ for each $a\in\A$; 
\item the resolvent of $\D$ is \emph{locally compact}, i.e.\ $\pi(a) (\D\pm i)^{-1}$ is compact for each $a\in A$.
\end{enumerate}
If no confusion arises, we will usually write $(\A,E_B,\D)$ instead of $(\A,{}_{\pi}E_B,\D)$ and $a$ instead of $\pi(a)$. 
If $B=\C$ and $A$ is trivially graded, we will write $E=\mH$ and refer to $(\A,\mH,\D)$ as a \emph{spectral triple} over $A$ (see \cite{Connes94}). 
\end{defn}

Given an \emph{unbounded Kasparov $A$-$B$-module} $(\A,E_B,\D)$, the `bounded transform' $F_\D := \D(1+\D^2)^{-\frac12}$ defines a (bounded) Kasparov module $(A,E_B,F_\D)$, representing a class $[\D] := [(A,E_B,F_\D)] \in \KK(A,B)$. 

An \emph{odd} unbounded Kasparov $A$-$B$-module $(\A,{}_{\pi}E_B,\D)$ is defined in the same way as above, except that $A$, $B$, and $E$ are assumed to be trivially graded, and $\D$ is not required to be odd. 
We can then consider the `even double' given by
\begin{align}
\label{eq:even_double}
&\left( A\otimes\CCliff_1 , \til E_B := (E\oplus E)_B , \til\D := \mattwo{0}{-i\D}{i\D}{0} \right) , & 
e &= \mattwo{0}{1}{1}{0} ,
\end{align} 
where $e$ denotes the generator of $\CCliff_1$. Then $(\A,E_B,\D)$ represents a class $[\D] := [(A\otimes\CCliff_1,\til E_B,F_{\til\D})] \in \KK^1(A,B)$.

\subsubsection{The unbounded Kasparov product}

Let $A$ be a separable $C^*$-algebra, and let $B,C$ be $\sigma$-unital $C^*$-algebras. 
There exists a bilinear associative pairing $\KK(A,B) \times \KK(B,C) \to \KK(A,C)$ \cite{Kas80b}. We refer to this pairing as the (internal) \emph{Kasparov product} over $B$. Given two $\KK$-classes $[\D_1] \in \KK(A,B)$ and $[\D_2] \in \KK(B,C)$, the Kasparov product is denoted by $[\D_1]\otimes_B[\D_2]$. The Kasparov product can be extended to a map $\KK(A_1,C_1\otimes B) \times \KK(B\otimes A_2,C_2) \to \KK(A_1\otimes A_2,C_1\otimes C_2)$ (where $A_1,A_2$ are assumed to be separable, and $B,C_1,C_2$ are $\sigma$-unital). 

The $\KK$-groups satisfy the formal periodicity $\KK(A\otimes\CCliff_2,B) \simeq \KK(A,B)$. Since the graded tensor product $\CCliff_1\otimes\CCliff_1$ equals $\CCliff_2$, the Kasparov product of two \emph{odd} $\KK$-classes yields a map
\begin{align*}
\KK^1(A,B) \times \KK^1(B,C) 
&\simeq \KK(A\otimes\CCliff_1,B) \times \KK(B\otimes\CCliff_1,C) 
\to \KK(A\otimes\CCliff_2,C) \simeq \KK(A,C) .
\end{align*}

By \cite[\S5, Theorem 7]{Kas80b}, we have the \emph{Bott periodicity} isomorphism 
$
\beta\colon\KK^1(A,C_0(\R,B)) \to \KK^0(A,B) . 
$
Consider the standard spectral triple over $C_0(\R)$ given by $(C_0(\R) , L^2(\R) , -i\partial_x)$. The Bott periodicity isomorphism can be implemented by taking the Kasparov product with the class $[-i\partial_x]$, i.e.\ for any class $\alpha\in\KK^1(A,C_0(\R,B))$ we have $\alpha \otimes_{C_0(\R)} [-i\partial_x] = \beta(\alpha)$ (indeed, the class $[-i\partial_x]$ is equal to the class $\alpha_1$ used in Kasparov's proof of Bott periodicity \cite[\S5, Theorem 7]{Kas80b}). 

The following theorem by Kucerovsky gives sufficient conditions for when an unbounded Kasparov module represents the Kasparov product. 

\begin{thm}[{\cite[Theorem 13]{Kuc97}}]
\label{thm:Kucerovsky}
Let $(\A, {}_{\pi_1}(E_1)_B,\D_1)$ and $(\B,{}_{\pi_2}(E_2)_C,\D_2)$ be unbounded Kasparov modules. 
Suppose that $(\A,{}_{\pi_1\hot\id}(E_1\hot_BE_2)_C,\D)$ is an unbounded Kasparov module such that:
\begin{enumerate}
\item for all $\psi$ in a dense subspace of $\pi_1(A)E_1$, the graded commutators 
$$
\left[\mattwo{\D}{0}{0}{\D_2},\mattwo{0}{T_{\psi}}{T_{\psi}^*}{0}\right]_\pm
$$
are bounded on $\Dom(\D\oplus \D_2)\subset (E_1\hot_BE_2)\oplus E_2$, where $T_{\psi}\colon E_2 \rightarrow E_1\hot_BE_2$ is given by $T_{\psi}(\xi) = \psi\hot\xi$;
\item we have the domain inclusion $\Dom\D\subset\Dom\D_1\hot1$;
\item there exists $c\in\R$ such that $\la(\D_1\hot1)\psi \mvert \D \psi\ra + \la\D \psi \mvert (\D_1\hot1)\psi\ra \geq c \la\psi|\psi\ra$ for all $\psi\in\Dom(\D)$. 
\end{enumerate}
Then $(\A,{}_{\pi_1\hot\id}(E_1\hot_BE_2)_C,\D)$ represents the Kasparov product of $(\A, {}_{\pi_1}(E_1)_B,\D_1)$ and $(\B,{}_{\pi_2}(E_2)_C,\D_2)$. 
\end{thm}

Let $\phi\colon B\to C$ be a $C^*$-algebra homomorphism, and let $[\phi]\in\KK^0(B,C)$ be the corresponding class represented by the (even) unbounded Kasparov $B$-$C$-module $(B,{}_\phi C_C,0)$. 
The homomorphism $\phi$ induces maps $\phi_*\colon\KK^\bullet(A,B)\to\KK^\bullet(A,C)$ and $\phi^*\colon\KK^\bullet(C,A)\to\KK^\bullet(B,A)$ given by
\begin{align*}
\phi_*(x) &= x \otimes_B [\phi] , & 
\phi^*(y) &= [\phi] \otimes_C y ,
\end{align*}
for any $x\in\KK^\bullet(A,B)$ and $y\in\KK^\bullet(C,A)$. 
The following lemma will be useful to us later. 

\begin{lem}
\label{lem:Kasp_prod_homom}
Let $A$, $B$, and $C$ be (trivially graded) $C^*$-algebras. 
Let $(\A,E_B,\D)$ be an odd unbounded Kasparov $A$-$B$-module, representing the class $[\D]\in\KK^1(A,B)$. 
Then the odd unbounded Kasparov $A$-$C$-module
\[
\left( \A , E \otimes_\phi C , \D\otimes1 \right) 
\]
represents the Kasparov product of $[\D]$ with $[\phi]$. 
\end{lem}
\begin{proof}
Using \cite[Propositions 4.7 \& 9.10]{Lance95}, one checks that $\D\otimes1$ indeed defines an odd unbounded Kasparov module. To show that $[\D\otimes1]$ represents the Kasparov product, we apply Kucerovsky's \cref{thm:Kucerovsky} to the `even doubles'. The second and third conditions are obviously satisfied, while the first condition holds because $T_{\D\psi}$ is bounded for any $\psi\in\Dom\D$. 
\end{proof}

\subsubsection{Symmetric elliptic operators}
\label{sec:symm_ell}

Let $\D$ be a symmetric elliptic first-order differential operator on a (possibly $\Z_2$-graded) vector bundle $\bF$ over a Riemannian manifold $M$. It is described in \cite[\S10.8]{Higson-Roe00} how $\D$ defines a $\K$-homology class $[\D]$ in $\KK^0(C_0(M),\C)$ or $\KK^1(C_0(M),\C)$ (depending on whether $\D$ is odd or even). If $\D$ is self-adjoint, this class $[\D]$ agrees with the usual $\K$-homology class of the spectral triple given by $\D$. 

For any open subset $U\subset M$, let $\D|_U$ denote (the closure of) the restriction of $\D$ to the initial domain $\Gamma_c^\infty(U,\bF|_U)$. We then obtain a well-defined class $[\D|_U] \in \KK^1(C_0(U),\C)$. Let $\iota_U\colon C_0(U)\into C_0(M)$ denote the obvious inclusion. 
\begin{prop}[{\cite[Proposition 10.8.8]{Higson-Roe00}}]
\label{prop:D_2_res}
We have the equality $[\D|_U] = \iota_U^*[\D] = [\iota_U] \otimes_{C_0(M)} [\D]$. 
\end{prop}
In particular, the class $\iota_U^*[\D]$ is defined intrinsically on $U$.

\subsection{Fredholm operators}

Let $B$ be a (trivially graded) $\sigma$-unital $C^*$-algebra, and let $E$ be a (possibly $\Z_2$-graded) Hilbert $B$-module. 
In \cref{sec:KK_sf}, we will consider a (suitably continuous) family of regular self-adjoint operators $\{\D(x)\}_{x\in X}$ on $E$ parametrised by a locally compact Hausdorff space $X$, such that we obtain a regular self-adjoint operator $\D(\cdot)$ on the Hilbert $C_0(X,B)$-module $C_0(X,E)$. We would like to associate to $\D(\cdot)$ a class in $\KK^1(\C,C_0(X,B))$, without assuming that $\D(\cdot)$ has compact resolvents (i.e.\ $\D(\cdot)$ does \emph{not} define an unbounded Kasparov module). Instead we only assume that $\D(\cdot)$ is Fredholm, following the approach of \cite[\S2]{Wah07}. 
In this subsection, we recall the notion of Fredholm operators on Hilbert modules, and we describe how a Fredholm operator determines a class in $\KK$-theory. 

\begin{defn}[{\cite[Definition 2.1]{Joa03}}]
Let $\D$ be a regular operator on a Hilbert $B$-module $E$. A \emph{right parametrix} for $\D$ is an adjointable endomorphism $Q_R \in \End_B(E)$ such that the  closed operator $\D Q_R$ is adjointable, and $\D Q_R - 1$ is compact. Similarly, $Q_L\in\End_B(E)$ is a \emph{left parametrix} if $Q_L\D$ is closable, $\bar{Q_L\D}$ is adjointable, and $\bar{Q_L\D} - 1$ is compact. We say $Q$ is a \emph{parametrix} for $\D$ if $Q$ is both a left and a right parametrix. The operator $\D$ is called \emph{Fredholm} if there exists a parametrix. 
\end{defn}

Given a regular operator $\D$ on a Hilbert $B$-module $E$, the following statements are equivalent:
\begin{enumerate}
\item $\D$ is Fredholm;
\item $\D\colon W\to E$ is a bounded Fredholm operator, where $W := \Dom\D$ is a Hilbert $B$-module equipped with the graph norm of $\D$; 
\item the bounded transform $\D(1+\D^*\D)^{-\frac12}$ on $E$ is Fredholm.
\end{enumerate}
The equivalence of 1) and 3) is proven in \cite[Lemma 2.2]{Joa03}, and the equivalence of 2) and 3) follows because $(1+\D^*\D)^{-\frac12}\colon E\to W$ is a unitary map. 

Let $\FEnd_B(E)$ denote the set of adjointable Fredholm operators on a Hilbert $B$-module $E$. By Kasparov stabilisation \cite[Theorem 13.6.2]{Blackadar98}, we have $E\oplus\mH_B \simeq \mH_B$, so we can embed $\FEnd_B(E) \into \FEnd_B(\mH_B)$ via the map $F \mapsto F\oplus1$. 
Let $q$ denote the quotient map of $\End_B(\mH_B)$ onto $\End_B(\mH_B) / \End_B^0(\mH_B)$. Then $q(F\oplus1)$ is invertible for any $F\in\FEnd_B(E)$, and therefore we obtain an odd $\K$-theory class $[q(F\oplus1)] \in \K_1\big(\End_B(\mH_B) / \End_B^0(\mH_B)\big)$. 
Since $\K_0(\End_B(\mH_B)) = \K_1(\End_B(\mH_B)) = 0$, 
the connecting map 
gives an isomorphism \cite[\S8.3]{Blackadar98} 
\[
\partial\colon \K_1\big(\End_B(\mH_B) / \End_B^0(\mH_B)\big) \to \K_0\big(\End_B^0(\mH_B)\big) \simeq \K_0(B) .
\]
We then define the index of $F\in\FEnd_B(E)$ by (cf.\ \cite[\S1.12]{Min87})
\[
\Index F := \partial[q(F\oplus1)] \in \K_0(B) .
\]
For the standard Hilbert $B$-module $\mH_B$, the index map $\Index\colon\FEnd_B(\mH_B)\to\K_0(B)$ is surjective. 
Writing $[\FEnd_B(\mH_B)]$ for the group of path components of $\FEnd_B(\mH_B)$, the index induces an isomorphism $\Index\colon[\FEnd_B(\mH_B)]\to\K_0(B)$. 
For more details, including an alternative definition of the Fredholm index which is more analogous to the usual Fredholm theory on Hilbert spaces, we refer to \cite{Min87} and \cite[Ch.\ 17]{Wegge-Olsen93}. 

Given a regular Fredholm operator $\D$ on $E_B$, we define $\Index(\D)$ as the index of $\D(1+\D^*\D)^{-\frac12}$. 

\begin{prop}[{\cite[\S2]{Wah07}}]
\label{prop:Fredholm_sa}
Let $\D$ be a regular self-adjoint operator on a Hilbert $B$-module $E$. Then the following statements are equivalent:
\begin{enumerate}
\item $\D$ is Fredholm;
\item there is $\epsilon>0$ such that for all $\phi\in C_c(-\epsilon,\epsilon)$ we have that $\phi(\D)\in\End_B^0(E)$;
\item there is $\epsilon>0$ such that for any continuous $\chi\colon\R\to\R$ with $\chi|_{(-\infty,-\epsilon]} = -1$ and $\chi|_{[\epsilon,\infty)} = 1$ we have $\chi(\D)^2-1\in\End_B^0(E)$.
\end{enumerate}
\end{prop}

\begin{defn}[{\cite[Definition 2.3]{Wah07}}]
A \emph{normalising function} for a regular self-adjoint operator $\D$ on a Hilbert $B$-module $E$ is an odd non-decreasing smooth function $\chi\colon\R\to\R$ with $\chi(0) = 0$, $\chi'(0) > 0$, $\lim_{x\to\infty} \chi(x) = 1$, and $\chi(\D)^2 - 1 \in \End_B^0(E)$. 
\end{defn}

From \cref{prop:Fredholm_sa}, it follows that a regular self-adjoint Fredholm operator on a Hilbert module admits many normalising functions. 

\begin{defn}
Let $\D_0$ and $\D_1$ be (even or odd) regular self-adjoint Fredholm operators on $\Z_2$-graded Hilbert $B$-modules $E_0$ and $E_1$, respectively. 
A \emph{homotopy} between $(E_0,\D_0)$ and $(E_1,\D_1)$ is an (even or odd) regular self-adjoint Fredholm operator $\til\D$ on a $\Z_2$-graded Hilbert $C([0,1],B)$-module $\til E$ such that $\ev_j(\til E,\til\D) \simeq (E_j,\D_j)$ (for $j=0,1$). 
Here $\simeq$ denotes unitary equivalence, and $\ev_t(\til E,\til\D) := (\til E\hot_{\rho_t}B,\til\D\hot1)$ (with $\rho_t$ as in \cref{defn:Kasp_mod}). 
We say that $\D_0$ and $\D_1$ are \emph{homotopic} if there exists a homotopy between $(E_0,\D_0)$ and $(E_1,\D_1)$. 
\end{defn}

\begin{prop}[cf.\ {\cite[Definition 2.4]{Wah07}}]
\label{prop:Fred_KK}
An odd resp.\ even, regular self-adjoint Fredholm operator $\D$ on a $\Z_2$-graded Hilbert $B$-module $E$ yields a well-defined class $[\D] := [\chi(\D)]$ in $\KK^0(\C,B)$ resp.\ $\KK^1(\C,B)$, where $\chi$ is any normalising function for $\D$.
Furthermore, if two such operators $\D$ and $\D'$ are homotopic, then $[\D]=[\D']$. 
\end{prop}
\begin{proof}
Let $\chi$ be any normalising function for $\D$. By assumption, $\chi(\D)^2-1$ is a compact endomorphism. Since $\chi(\D)$ is also self-adjoint, we see that $(E_B,\chi(\D))$ is a (bounded) Kasparov $\C$-$B$-module (since any normalising function is odd, we note that $\chi(\D)$ is odd whenever $\D$ is odd). 

We need to show that $[\chi(\D)]$ is independent of the choice of $\chi$. 
By \cref{prop:Fredholm_sa}, there exists an $\epsilon>0$ such that $\phi(\D)$ is compact for any $\phi\in C_c(-\epsilon,\epsilon)$. 
Pick a positive function $\phi_0\in C_c(-\epsilon,\epsilon)$ such that $\phi_0(x)=1$ for all $x\in(-\frac\epsilon2,\frac\epsilon2)$. Define $\phi_+ := (1-\phi_0)H$ and $\phi_- := (1-\phi_0)(1-H)$, where $H$ denotes the Heaviside function (i.e., the characteristic function of $[0,\infty)$). Furthermore, define $\chi_-,\chi_0,\chi_+\in C^\infty(\R)$ by $\chi_- := \chi \phi_-$, $\chi_0 := \chi \phi_0$, and $\chi_+ := \chi \phi_+$. 
Then $\chi_0(\D)$ is compact. 
If $\chi'$ is another normalising function for $\D$, then $\chi'_0(\D)$ is also compact. 
Consider $\chi_\pm^t(\D) := \pm \sqrt{(1-t)\chi_\pm(\D)^2 + t\chi_\pm'(\D)^2}$. For each $t\in[0,1]$ we have that 
\begin{align*}
\chi_\pm^t(\D)^2 - \phi_\pm(\D)^2 &= (1-t)\chi_\pm(\D)^2 + t\chi_\pm'(\D)^2 - \phi_\pm(\D)^2 \\
&= (1-t) (\chi(\D)^2-1) \phi_\pm(\D)^2 + t (\chi'(\D)^2-1) \phi_\pm(\D)^2
\end{align*}
is compact. 
Finally, define $\chi^t(\D) := \chi_-^t(\D) + \chi_+^t(\D) + (1-t)\chi_0(\D) + t\chi_0'(\D)$. 
Since $\chi_0(\D)$, $\chi_0'(\D)$, and $1-\phi_+(\D)^2-\phi_-(\D)^2$ are compact, and since $\chi_+^t(\D)\chi_-^t(\D)=0$, we see that $\chi^t(\D)^2-1$ is equal to $\chi_-^t(\D)^2 + \chi_+^t(\D)^2 - \phi_-(\D)^2 - \phi_+(\D)^2$ modulo compact operators, and therefore $\chi^t(\D)^2-1$ is compact. 
Hence $\chi^t(\D)$ provides a homotopy (in the sense of Kasparov) between $\chi(\D)$ and $\chi'(\D)$, which shows that the class $[\chi(\D)]$ is independent of the choice of $\chi$. 

Finally, let $(\til E,\til\D)$ be a homotopy between $(E,\D)$ and $(E',\D')$, and let $\chi$ be a normalising function for $\til\D$. Then $\chi(\til\D)$ is a homotopy (in the sense of Kasparov) for $\chi(\D)$ and $\chi(\D')$, and therefore $[\D]=[\D']$. 
\end{proof}

Consider the special case when $\D$ is not only Fredholm, but in fact has compact resolvents. Then $(\C,E_B,\D)$ is an unbounded Kasparov module. In this case, the `bounded transform' $b\colon x\mapsto x(1+x^2)^{-\frac12}$ is a normalising function for $\D$. Hence the class $[\D]$ defined above agrees with the usual class associated to the unbounded Kasparov module $(\C,E_B,\D)$. 

We consider the following standard isomorphisms of $\KK^\bullet(\C,B)$ with $\K_\bullet(B)$ (see \cite[\S2.1]{Wah07}). 
In the even ($\Z_2$-graded) case, consider an odd Fredholm operator 
\[
\D = \mattwo{0}{\D_-}{\D_+}{0} 
\]
on the standard $\Z_2$-graded Hilbert module $\hat\mH_B = \mH_B\oplus\mH_B$. The standard isomorphism $\KK^0(\C,B) \to \K_0(B)$ maps the class $[\D]$ to the Fredholm index of $\D_+$. 
In the odd case, the standard isomorphism $\KK^1(\C,B) \to \K_1(\End_B^0(\mH_B)) \simeq \K_1(B)$ assigns to the class represented by a bounded Kasparov module $(\mH_B,T)$ the element $\left[ e^{i\pi(T+1)} \right] \in \K_1(\End_B^0(\mH_B))$. 
We will see in \cref{sec:KK_sf} how the class $[\D]\in\KK^1(\C,B)$ of an unbounded regular self-adjoint Fredholm operator is related to the spectral flow.

\subsection{Spectral flow and the Kasparov product}
\label{sec:KK_sf}

Let $X$ be a locally compact, paracompact space, and consider a family of operators $\{\D(x)\}_{x\in X}$ on a Hilbert $B$-module $E$. We define the operator $\D(\cdot)$ on the Hilbert $C_0(X,B)$-module $C_0(X,E)$ by
\[
\big(\D(\cdot)\psi\big)(x) := \D(x) \psi(x) ,
\]
for all $\psi$ in the domain 
\begin{align}
\label{eq:dom_fam}
\Dom\D(\cdot) := \big\{ \psi\in C_0(X,E) : \psi(x)\in\Dom\D(x) \text{ and } \D(\cdot)\psi\in C_0(X,E) \big\} .
\end{align}
In order for $\D(\cdot)$ to be a densely defined operator on $C_0(X,E)$, we of course need to assume that the family $\{\D(x)\}_{x\in X}$ is suitably continuous. The following lemma characterises the continuity required for regular self-adjoint operators. 

\begin{lem}[{\cite[Proposition 2.5]{Wah07}}]
\label{lem:reg_sa_cts_fam}
Consider a family of regular self-adjoint operators $\{\D(x)\}_{x\in X}$ on a Hilbert $B$-module $E$. 
Then the operator $\D(\cdot)$ on the Hilbert $C_0(X,B)$-module $C_0(X,E_B)$ is regular self-adjoint if and only if the resolvents $\{(\D(x)\pm i)^{-1}\}_{x\in X}$ are strongly continuous. 
\end{lem}

The notion of spectral flow for a path of self-adjoint operators (typically pa\-ra\-metrised by the unit interval) was first defined by Atiyah and Lusztig, and it appeared in the work of Atiyah, Patodi, and Singer \cite[\S7]{APS76}. Heuristically, the spectral flow of a path of self-adjoint Fredholm operators counts the net number of eigenvalues which pass through zero. An analytic definition of spectral flow was given by Phillips in \cite{Phi96}. 
For a path of Fredholm operators on a Hilbert module, the most general definition of spectral flow (to the author's knowledge) is given in \cite{Wah07}.\footnote{There are other notions of spectral flow which are based on different notions of the Fredholm property, e.g.\ replacing the Fredholm property by invertibility up to some closed ideal (not necessarily the compacts) \cite{KNR12}, but here we only consider the case of Fredholm operators.} 
\begin{defn}[{\cite[Definition 3.9]{Wah07}}]
Let $X$ be a locally compact, paracompact space, and consider a regular self-adjoint operator $\D(\cdot) = \{\D(x)\}_{x\in X}$ on the Hilbert $C_0(X,B)$-module $C_0(X,E)$. 
We say there exist \emph{locally trivialising families} for $\D(\cdot)$ if for each $x\in X$ there exist a precompact open neighbourhood $O_x$ of $x$ and a family of bounded operators $\{A(y)\}_{y\in O_x}$ such that $A(\cdot)$ is self-adjoint on the Hilbert $C_0(O_x,B)$-module $C_0(O_x,E)$, $A(\cdot)\colon\Dom\D(\cdot)|_{O_x}\to C_0(O_x,E)$ is compact, and $\D(\cdot)|_{O_x}+A(\cdot)$ is invertible.
\end{defn}
It follows that $(\D(\cdot)|_{O_x}+A(\cdot))^{-1}$ is a parametrix for $\D(\cdot)|_{O_x}$, so $\D(\cdot)|_{O_x}$ is Fredholm. If $X$ is compact, the existence of locally trivialising families for $\{\D(x)\}_{x\in X}$ then implies that $\D(\cdot)$ is Fredholm. 

Now consider $X = [0,1]$. 
The spectral flow of $\{\D(x)\}_{x\in[0,1]}$ is then defined and depends (in principle) on trivialising operators $A(0)$ of $\D(0)$ and $A(1)$ of $\D(1)$. 
If the endpoints $\D(0)$ and $\D(1)$ are invertible, we can take $A(0)=A(1)=0$ to obtain a canonical definition of spectral flow, i.e.\ we have a map $\SF$ which assigns, to the family $\{\D(x)\}_{x\in[0,1]}$, a class $\SF(\{\D(x)\}_{x\in[0,1]}) \in \K_0(B)$ in the even $\K$-theory of $B$. 
We refer to \cite[Definition 3.10]{Wah07} for the definition of this spectral flow. 

\begin{lem}
\label{lem:loc_triv_Fred}
Let $X$ be a locally compact, paracompact space. Consider a regular self-adjoint operator $\D(\cdot) = \{\D(x)\}_{x\in X}$ on the Hilbert $C_0(X,B)$-module $C_0(X,E)$ for which locally trivialising families exist. 
Suppose that there exists a compact set $K\subset X$ such that $\D(\cdot)|_{X\backslash K}$ is Fredholm. 
Then $\D(\cdot)$ is Fredholm. 
\end{lem}
\begin{proof}
By assumption, for each $x\in X$ there exist a precompact open neighbourhood $O_x$ of $x$ and a trivialising operator $A(\cdot)$ on $C_0(O_x,E)$ such that $\D(\cdot)|_{O_x}+A(\cdot)$ is invertible. 
Then $(\D(\cdot)|_{O_x}+A(\cdot))^{-1}$ is a parametrix for $\D(\cdot)|_{O_x}$, so $\D(\cdot)|_{O_x}$ is Fredholm. 
Using compactness of $K$, we have a finite open refinement $\{U_n\}_{n=1}^N$ of $\{O_x\}_{x\in K}$, such that $K\subset \bigcup_{n=1}^N U_n$. 
For each $n$ we have a trivialising family $\{A_n(x)\}_{x\in U_n}$ as above, and we write $Q_n := (\D(\cdot)|_{U_n}+A_n(\cdot))^{-1}$. 
Let $U_0 := X\backslash K$, and let $Q_0$ be a parametrix for $\D(\cdot)|_{X\backslash K}$. 
Let $\chi_n$ be a partition of unity for the finite open cover $\{U_n\}_{n=0}^N$ of $X$. 
Then $Q := \sum_{n=0}^N Q_n \chi_n$ is a parametrix for $\D(\cdot)$. 
\end{proof}

Thus, if there exist locally trivialising families for $\D(\cdot)$, and $\D(\cdot)$ is `Fredholm near infinity', then $\D(\cdot)$ is Fredholm and therefore defines a class $[\D(\cdot)]\in\KK^1(\C,C_0(\R,B)) \simeq \KK^0(\C,B) \simeq \K_0(B)$ in the $\K$-theory of $B$. 
This suggests the following generalisation of the notion of spectral flow. 
\begin{defn}
Let $X$ be a locally compact, paracompact space, and let $\D(\cdot)$ be a regular self-adjoint Fredholm operator on $C_0(X,E_B)$. 
Then we define the \emph{$\KK$-theoretic spectral flow} of $\D(\cdot)$ as 
\[
\ASF(\D(\cdot)) := [\D(\cdot)] \in \KK^1(\C,C_0(X,B)) \simeq \K_1(C_0(X,B)) .
\]
\end{defn}
In the special case $X=\R$, and using
the Bott periodicity $\beta\colon\K_1(C_0(\R,B))\xrightarrow{\simeq}\K_0(B)$, we can view the $\KK$-theoretic spectral flow of $\D(\cdot)$ as an element in the even $\K$-theory of $B$. 
In this case, the $\KK$-theoretic spectral flow agrees with the usual notion of spectral flow under the Bott periodicity isomorphism. 
This was already shown by Wahl \cite{Wah07} for the spectral flow of a family on the unit interval. 

We point out that the spectral flow is also well-defined on the whole real line $\R$, as long as the family of operators is uniformly invertible outside a bounded interval $[a,b]\subset\R$ (i.e.\ $\sup_{x\in \R\backslash[a,b]}\|\D(x)^{-1}\| < \infty$). 
Indeed, in this case we can simply define the spectral flow by restricting to this bounded interval: 
$\SF\big(\{\D(x)\}_{x\in\R}\big) := \SF\big(\{\D(x)\}_{x\in[a,b]}\big)$. 
This is well-defined, because the invertible part of the family does not contribute to the spectral flow. 

\begin{prop}
\label{prop:ASF_R}
Let $\D(\cdot)$ be a regular self-adjoint operator on the Hilbert $C_0(\R,B)$-module $C_0(\R,E)$. Suppose that $\D(x)$ is uniformly invertible outside a bounded interval $[a,b]\subset\R$, and that there exist locally trivialising families for $\{\D(x)\}_{x\in\R}$. 
Then $\D(\cdot)$ is Fredholm, and we have the equality
\[
\beta\big(\ASF(\D(\cdot))\big) = \SF\big(\{\D(x)\}_{x\in\R}\big) \in \K_0(B) .
\]
\end{prop}
\begin{proof}
From \cref{lem:loc_triv_Fred} we know that $\D(\cdot)$ is Fredholm, so the $\KK$-theoretic spectral flow is well-defined. 
We know that $\D(x)$ is uniformly invertible outside a bounded interval $[a,b]\subset\R$. Without loss of generality, we may assume that $[a,b]=[0,1]$. Consider then the homotopy
\[
\D^t(x) := 
\begin{cases}
\D(tx) & x<0 , \\
\D(x) & 0\leq x\leq 1 , \\
\D(t(x-1)+1) & x>1 .
\end{cases}
\]
Since $\D(\cdot)$ is regular self-adjoint, we know from \cref{lem:reg_sa_cts_fam} that the resolvents $(\D(x)\pm i)^{-1}$ are strongly continuous. By construction, the resolvents $(\D^t(x)\pm i)^{-1}$ are then strongly continuous as well (jointly in $t$ and $x$). Hence we know (again from \cref{lem:reg_sa_cts_fam}) that $\D^\bullet(\cdot)$ is a regular self-adjoint operator on $C_0([0,1]\times\R,E)$. 
Let $Q$ be a parametrix for $\D(\cdot)$. 
Then $\D^\bullet(\cdot)$ is also Fredholm, because using a partition of unity we can `patch together' the parametrix $Q^\bullet = Q$ on $[0,1]\times[0,1]$ with the inverse of $\D^\bullet(\cdot)$ outside of $[0,1]\times[0,1]$. 
Hence $\D^\bullet(\cdot)$ indeed gives a homotopy between $\D^0(\cdot)$ and $\D^1(\cdot)=\D(\cdot)$, and we have reduced the problem to the case where $\D(\cdot)$ is constant outside a bounded interval. For this case, it has been shown in \cite[Proposition 4.2 \& Theorem 4.4]{Wah07} that $\beta([\D^0(\cdot)]) = \SF(\{\D(x)\}_{x\in[0,1]}) \in \K_0(B)$. 
\end{proof}

\begin{example}
\label{eg:sf_cyl}
Suppose that $B=\C$ (so $E=\mH$ is a Hilbert space) and $X = \R\times Y$, where $Y$ is a compact Hausdorff space. 
For simplicity, let us assume that $\pS(r,y) = \pS(0,y)$ for all $r\leq0$, and $\pS(r,y) = \pS(1,y)$ for all $r\geq1$. 
Using the isomorphism $C_0(X) \simeq C_0(\R,C(Y))$, and writing $\pS^y(r) := \pS(r,y)$, we can alternatively view the operators $\pS(r,y)$ as defining a family $\{\pS^\bullet(r)\}_{r\in\R}$ of regular self-adjoint operators on the Hilbert $C(Y)$-module $C(Y,\mH)$. 
Assuming there exist locally trivialising families for $\{\pS^\bullet(r)\}_{r\in\R}$, we see from \cref{prop:ASF_R} that the $\KK$-theoretic spectral flow $\ASF\big(\pS(\cdot,\cdot)\big) \in \KK^1(\C,C_0(X))$ is equal (under the Bott periodicity isomorphism) to the (usual) spectral flow $\SF\big(\{\pS^\bullet(r)\}_{r\in[0,1]}\big) \in \K_0(C(Y))$. 
\end{example}

We would like to describe the spectral flow in terms of the Kasparov product. Consider the standard spectral triple $(C_0^1(\R), L^2(\R), -i\partial_x)$ on the real line, whose $\K$-homology class $[-i\partial_x]$ is the generator of $\KK^1(C_0(\R),\C) \simeq \Z$. 

\begin{prop}
\label{prop:SF_Kasp_prod}
Given a regular self-adjoint Fredholm operator $\D(\cdot)$ on the Hilbert $C_0(\R)$-module $C_0(\R,E)$, we have
\[
\beta\big(\ASF(\D(\cdot))\big) = [\D(\cdot)] \otimes_{C_0(\R)} [-i\partial_x] \in K_0(B) .
\]
If $\D(x)$ is uniformly invertible outside a bounded interval $[a,b]\subset\R$, and there exist locally trivialising families for $\{\D(x)\}_{x\in\R}$, then
\[
\SF\big(\{\D(x)\}_{x\in\R}\big) = [\D(\cdot)] \otimes_{C_0(\R)} [-i\partial_x] \in K_0(B) .
\]
\end{prop}
\begin{proof}
For the first statement, we observe that the Bott periodicity isomorphism $\beta\colon\KK^1(A,C_0(\R,B)) \to \KK^0(A,B)$ is implemented by taking the Kasparov product with the class $[-i\partial_x]$ (see \cref{sec:KK}). 
The second statement then follows from \cref{prop:ASF_R}. 
\end{proof}

\section{Dirac-Schr\"odinger operators}
\label{sec:DS}

\subsection{Families of unbounded operators}

Let $X$ be a locally compact, paracompact, topological space, let $B$ be a (trivially graded) $\sigma$-unital $C^*$-algebra, and let $E$ be a countably generated Hilbert $B$-module. 
Let $\{\pS(x)\}_{x\in X}$ be a family of regular self-adjoint operators on $E$ satisfying the following assumptions. 
\begin{enumerate}
\item[(a1)]
\customlabel{ass:a1}{(a1)}
The domain $W := \Dom\pS(x)$ is independent of $x\in X$, and the inclusion $W\into E$ is compact (where $W$ is viewed as a Hilbert $B$-module equipped with the graph norm of $\pS(x_0)$, for some $x_0\in X$). 
\item[(a2)]
\customlabel{ass:a2}{(a2)}
The map $\pS\colon X\to\Hom_B(W,E)$ is norm-continuous. 
\item[(a3)]
\customlabel{ass:a3}{(a3)}
There exists a compact subset $K\subset X$ such that $\pS(x)$ is uniformly invertible on $X\backslash K$. 
\end{enumerate}
Here we say that $\pS(x)$ is uniformly invertible on $X\backslash K$ if $\pS(x)$ is invertible for all $x\in X\backslash K$ and we have a uniform bound $\sup_{x\in X\backslash K}\|\pS(x)^{-1}\| < \infty$. 
Furthermore, we say that the graph norms of $\pS(x)$ are uniformly equivalent (to the norm of $W$) if there exist constants $C_1,C_2>0$ such that $C_1\|\xi\|_W \leq \|\xi\|_{\pS(x)} \leq C_2\|\xi\|_W$ for all $\xi\in W$ and all $x\in X$. 

\begin{lem}
\label{lem:cts_resolvent}
The resolvent operators $(\pS(x)\pm i)^{-1}\in\End_B(E)$ depend norm-continuously on $x\in X$.
\end{lem}
\begin{proof}
Let $x\in X$, and consider $W_x := \Dom\pS(x)$ as a Hilbert $B$-module equipped with the graph norm. 
Since $W_x = W = \Dom\pS(x_0)$, we know from \cref{lem:dom_inc_rel_bdd} that the graph norm of $\pS(x)$ is equivalent to the norm of $W$. 
Therefore we know that $\pS$ is also continuous as a map $X\to\Hom_B(W_x,E)$. In particular, $(\pS(y)-\pS(x)) (\pS(x)-i)^{-1}$ depends continuously on $y$. The statement then follows from the inequality
\begin{align*}
\big\| (\pS(x)\pm i)^{-1} - (\pS(y)\pm i)^{-1} \big\| 
&\leq \big\| (\pS(y)\pm i)^{-1} \big\| \; \big\| ( \pS(y) - \pS(x) ) (\pS(x)\pm i)^{-1} \big\| .
\qedhere
\end{align*}
\end{proof}

\begin{lem}
\label{lem:pS_reg-sa}
The family $\{\pS(x)\}_{x\in X}$ gives an essentially regular self-adjoint operator $\pS(\cdot)$ on the Hilbert $C_0(X,B)$-module $C_0(X,E)$ given by
\[
\big(\pS(\cdot)\psi\big)(x) := \pS(x)\psi(x) ,
\]
for all $\psi$ in the initial domain $C_c(X,W)$. 
Furthermore, if the graph norms of $\{\pS(x)\}_{x\in X}$ are uniformly equivalent, then the closure of $\pS(\cdot)$ is regular self-adjoint on the domain $\Dom\pS(\cdot) = C_0(X,W)$. 
\end{lem}
\begin{proof}
By \cref{lem:cts_resolvent}, the resolvents $(\pS(x)\pm i)^{-1}$ are norm-continuous, so we already know from \cref{lem:reg_sa_cts_fam} that (the closure of) $\pS(\cdot)$ is regular self-adjoint, with the domain given in \cref{eq:dom_fam}. We need to check that $C_c(X,W)$ is a core for $\pS(\cdot)$. 

Since $W$ is dense in $E$, $C_c(X,W)$ is dense in $C_0(X,E)$. For $\psi\in C_c(X,W)$ we have 
\begin{align*}
\|\pS(x)\psi(x)-\pS(y)\psi(y)\| &\leq \|(\pS(x)-\pS(y))\psi(x)\| + \|\pS(y)(\psi(x)-\psi(y))\| \\
&\leq \|\pS(x)-\pS(y)\|_{W\to E} \; \|\psi(x)\|_W + \|\pS(y)\|_{W\to E} \; \|\psi(x)-\psi(y)\|_W .
\end{align*}
Since $\psi\in C_c(X,W)$ is continuous, and $\pS(x)$ depends norm-continuously on $x$ as a map $W\to E$, the above inequality shows that $\pS(x)\psi(x)\in E$ also depends norm-continuously on $x$. 
Hence $\pS(\cdot)\cdot C_c(X,W) \subset C_0(X,E)$, and therefore $C_c(X,W) \subset \Dom\pS(\cdot)$. 
Since $\pS(x)\pm i$ is surjective for each $x\in X$ and the resolvents $(\pS(x)\pm i)^{-1}$ are norm-continuous, it follows that $\pS(\cdot)\pm i\colon C_c(X,W)\to C_0(X,E)$ has dense range. Hence $C_c(X,W)$ is a core for $\pS(\cdot)$. 
Finally, if the graph norms of $\pS(x)$ are uniformly equivalent, it follows that the graph norm of $\pS(\cdot)$ is equivalent to the supremum-norm on $C_c(X,W)$, so that $\Dom\pS(\cdot) = C_0(X,W)$. 
\end{proof}

\begin{lem}
\label{lem:pS_loc-cpt-res}
The operator $\pS(\cdot)$ has locally compact resolvents, i.e.\ $f(\pS(\cdot)\pm i)^{-1}$ is compact for every $f\in C_0(X)$.
\end{lem}
\begin{proof}
We know from \cref{lem:cts_resolvent} that the family $(\pS(x)\pm i)^{-1}$ is norm-continuous. By assumption, $(\pS(x)\pm i)^{-1}$ is compact for each $x\in X$, so we find that $(\pS(\cdot)\pm i)^{-1} \in C_b(X,\End_B^0(E))$. Multiplying by $f\in C_0(X)$ we thus obtain $f(\pS(\cdot)\pm i)^{-1} \in C_0(X,\End_B^0(E)) = \End_{C_0(X,B)}^0(C_0(X,E))$. 
\end{proof}

\begin{prop}
\label{prop:pS_Fred}
The operator $\pS(\cdot)$ on the Hilbert $C_0(X,B)$-module $C_0(X,E)$ is regular self-adjoint and Fredholm, and hence it defines a class 
\[
[\pS(\cdot)] \in \KK^1(\C,C_0(X,B)) . 
\]
\end{prop}
\begin{proof}
We have seen in \cref{lem:pS_reg-sa} that $\pS(\cdot)$ is regular self-adjoint. 
By \cref{lem:pS_loc-cpt-res}, we know that $\pS(\cdot)$ has locally compact resolvents. Consider $\phi\in C_c(X)$ such that $\phi(x) = 1$ for all $x\in K$. Then $Q := (\pS(\cdot)-i)^{-1} \phi + \pS(\cdot)^{-1} (1-\phi)$ is a parametrix for $\pS(\cdot)$. 
\end{proof}

\begin{lem}
\label{lem:pS_loc_triv_fam}
Let $\mH$ be a separable Hilbert space, and consider a family of self-adjoint operators $\{\pS(x)\}_{x\in X}$ on $\mH$, satisfying assumptions \ref{ass:a1}-\ref{ass:a3}. 
Then there exist locally trivialising families for $\pS(\cdot)$. 
\end{lem}
\begin{proof}
Consider a point $x_0\in X$. Since $\pS(x_0)$ has compact resolvents, the spectral projection $A_0 := P_{[-1,1]}(\pS(x_0))$ is compact. The operator $\pS(x_0)+2A_0$ is invertible. By continuity, there exists a neighbourhood $O_0$ of $x_0$ such that $\pS(x)+2A_0$ is invertible for all $x\in O_0$. Hence $2A_0$ provides a (constant) trivialising family for $\{\pS(x)\}_{x\in O_0}$. 
\end{proof}
\begin{remark}
In the special case $X=\R$, the above lemma shows that the spectral flow of $\{\pS(x)\}_{x\in\R}$ is well-defined. From \cref{prop:ASF_R} we then know that the class $[\pS(\cdot)] \in \KK^1(\C,C_0(\R)) \simeq \Z$ is given by the spectral flow of $\{\pS(x)\}_{x\in\R}$. However, it is not clear if the statement of the above lemma remains valid for families of operators on Hilbert \emph{modules} (rather than Hilbert spaces). 
\end{remark}

In \cref{sec:Kasp_prod}, we will compute a Kasparov product with the class $[\pS(\cdot)] \in \KK^1(\C,C_0(X,B))$. 
However, the existing literature on the unbounded Kasparov product deals only with unbounded Kasparov modules. 
Although the Fredholm operator $\pS(\cdot)$ represents a class in $\KK^1(\C,C_0(X,B))$, in general it does not define an unbounded Kasparov $\C$-$C_0(X,B)$-module, because it might not have compact resolvents. Therefore we show next that we can replace $\pS(\cdot)$ by an operator which does have compact resolvents (following \cite[\S8]{KL13}), so that we can make use of existing results on the Kasparov product of unbounded Kasparov modules. 

\begin{lem}
\label{lem:pS_cpt_res}
Let $f\in C_0(X)$ be a strictly positive function vanishing at infinity, such that $f(x)=1$ for all $x\in K$. 
Then the operator $f^{-1}\pS(\cdot)$ corresponding to the family $\{f(x)^{-1}\pS(x)\}_{x\in X}$ defines an odd unbounded Kasparov module $(\C,C_0(X,E)_{C_0(X,B)},f^{-1}\pS(\cdot))$, and $[f^{-1}\pS(\cdot)] = [\pS(\cdot)] \in \KK^1(\C,C_0(X,B))$. 
\end{lem}
\begin{proof}
The compactness of the resolvent of $f^{-1}\pS(\cdot)$ is shown in \cite[Proposition 8.7]{KL13}. We obtain a homotopy between $\pS^0(\cdot) = \pS(\cdot)$ and $\pS^1(\cdot) = f^{-1}\pS(\cdot)$ from the operators $\pS^t(x) := f^{-t}(x) \pS(x)$ for $t\in[0,1]$. 
Indeed, since $f^{-1}$ is uniformly positive and continuous, the family $\{\pS^t(x)\}_{(t,x)\in[0,1]\times X}$ again satisfies assumptions \ref{ass:a1}-\ref{ass:a3}. By \cref{prop:pS_Fred}, the operator $\pS^\bullet(\cdot)$ on the Hilbert $C_0([0,1]\times X,B)$-module $C_0([0,1]\times X,E)$ is therefore regular self-adjoint and Fredholm. 
\end{proof}

Next we will show that the class $[\pS(\cdot)] \in \KK^1(\C,C_0(X,B))$ is completely determined by the family $\{\pS(x)\}_{x\in U}$ for any open neighbourhood $U$ of $K$. 
Let $\iota_U\colon C_0(U)\into C_0(X)$ denote the obvious inclusion. We associate to it the class $[\iota_U]\in\KK^0(C_0(U),C_0(X))$ represented by the module $(C_0(U) , {}_{\iota_U}C_0(X)_{C_0(X)} , 0 )$. 

\begin{lem}
\label{lem:pS_res}
Let $U$ be an open neighbourhood of $K$. Then 
\[
[\pS(\cdot)] = [\pS(\cdot)|_U] \otimes_{C_0(U)} [\iota_U] \in \KK^1(\C,C_0(X,B)) ,
\]
where $[\pS(\cdot)|_U] \in \KK^1(\C,C_0(U))$ is the class corresponding to the family $\{\pS(x)\}_{x\in U}$. 
\end{lem}
\begin{proof}
First, we note that the restricted family $\{\pS(x)\}_{x\in U}$ also satisfies assumptions \ref{ass:a1}-\ref{ass:a3}, so that $[\pS(\cdot)|_U]$ is well-defined. 
Let $f\in C_0(U)$ be a strictly positive function such that $f(x)=1$ for all $x\in K$. 
By \cref{lem:pS_cpt_res}, we know that $f^{-1}\pS(\cdot)|_U$ defines an unbounded Kasparov module and $[f^{-1}\pS(\cdot)|_U] = [\pS(\cdot)|_U]]$. 
From \cref{lem:Kasp_prod_homom} we know that the Kasparov product $[f^{-1}\pS(\cdot)|_U] \otimes [\iota_U]$ is represented by 
$(\C,C_0(U,E)\otimes_{\iota_U} C_0(X),f^{-1}\pS(\cdot)|_U\otimes1)$. Furthermore, $f^{-1}\pS(\cdot)|_U\otimes1$ is homotopic to the Fredholm operator $\pS(\cdot)|_U\otimes1$. 
We need to show that the latter operator is homotopic to $\pS(\cdot)$. 

Let $\til X$ be the subspace of $[0,1]\times X$ given by the union of $[0,1]\times U$ and $(0,1]\times X$. 
Define the Hilbert $C_0([0,1]\times X,B)$-module $\til E := C_0(\til X,E)$. 
Consider the operator $\pS^\bullet(\cdot)$ on $\til E$ 
given by $\pS^t(x) := \pS(x)$ for all $(t,x)\in\til X$. 
Then one easily sees that the family $\{\pS^t(x)\}_{(t,x)\in\til X}$ also satisfies the assumptions \ref{ass:a1}-\ref{ass:a3} (we note that this family is uniformly invertible outside the compact subset $[0,1]\times K\subset\til X$). Hence $\pS^\bullet(\cdot)$ is Fredholm by \cref{prop:pS_Fred}. 
We observe that the restriction of $\pS^\bullet(\cdot)$ to $\{0\}\times U$ acts precisely as $\pS(\cdot)|_U\otimes1$ on $C_0(U,E)\otimes_{\iota_U} C_0(X)$, while the restriction to $\{1\}\times X$ is simply $\pS(\cdot)$ on $C_0(X,E)$. Thus $\pS^\bullet(\cdot)$ is a homotopy between $\pS(\cdot)$ and $\pS(\cdot)|_U\otimes1$. 
\end{proof}

\subsection{The product operator}
\label{sec:product_op}

From here on, we make the following standing assumptions. 
\begin{standing}
\customlabel{ass:standing}{Standing Assumptions}
Let $B$ be a (trivially graded) $\sigma$-unital $C^*$-algebra, and let $E$ be a (trivially graded) countably generated Hilbert $B$-module. 
Let $M$ be a connected Riemannian manifold, and let $\D$ be an essentially self-adjoint elliptic first-order differential operator on a hermitian vector bundle $\bF\to M$.
Let $\{\pS(x)\}_{x\in M}$ be a family of regular self-adjoint operators on $E$ satisfying the following assumptions. 
\begin{enumerate}
\item[(A1)]
\customlabel{ass:A1}{(A1)}
The domain $W := \Dom\pS(x)$ is independent of $x\in M$, and the inclusion $W\into E$ is compact (where $W$ is viewed as a Hilbert $B$-module equipped with the graph norm of $\pS(x_0)$, for some $x_0\in M$). 
\item[(A2)]
\customlabel{ass:A2}{(A2)}
The map $\pS\colon M\to\Hom_B(W,E)$ is norm-continuous. 
\item[(A3)]
\customlabel{ass:A3}{(A3)}
There exists a compact subset $K\subset M$ such that $\pS(x)$ is uniformly invertible on $M\backslash K$. 
\item[(A4)]
\customlabel{ass:A4}{(A4)}
There exists a finite open cover $\{V_j\}$ of $M\backslash K$ with points $x_j\in\bar V_j$ and positive numbers $a_j<1$ such that $\big\| \big( \pS(x)-\pS(x_j) \big) \pS(x_j)^{-1} \big\| \leq a_j$ for all $x\in V_j$. 
\end{enumerate}
\end{standing}

\begin{lem}
\label{lem:uegn}
The graph norms of $\pS(x)$ are uniformly equivalent. 
\end{lem}
\begin{proof}
For any $x,y\in M$, since $\Dom\pS(x)=\Dom\pS(y)$, we know 
from \cref{lem:dom_inc_rel_bdd}
that the graph norms of $\pS(x)$ and $\pS(y)$ are equivalent. 
Using compactness of $K$ and continuity of $\pS(x)$, 
it follows that the graph norms of $\pS(x)$ are uniformly equivalent for all $x\in K$. For $x\in M\backslash K$, we know from assumption \ref{ass:A4} that $x\in V_j$ for some $j$. Then 
\[
\big\| (\pS(x)\pm i)(\pS(x_1)\pm i)^{-1} \big\| \leq (a_j+1) \big\| (\pS(x_j)\pm i)(\pS(x_1)\pm i)^{-1} \big\| .
\]
Using a Neumann series argument, we know that $\big\| \pS(x_j) \pS(x)^{-1} \big\| \leq \frac1{1-a_j}$. Therefore
\begin{multline*}
\big\| (\pS(x_1)\pm i)(\pS(x)\pm i)^{-1} \big\| \\
\leq \big\| (\pS(x_1)\pm i)(\pS(x_j)\pm i)^{-1} \big\| \; \big\| (\pS(x_j)\pm i) \pS(x_j)^{-1} \big\| \cdot \frac1{1-a_j} \cdot \big\| \pS(x)(\pS(x)\pm i)^{-1} \big\| .
\end{multline*}
These inequalities show that for $x\in M\backslash K$ the graph norms of $\pS(x)$ are uniformly equivalent to the graph norm of $\pS(x_1)$. 
\end{proof}

We consider the balanced tensor product $L^2(M,E\otimes\bF) := C_0(M,E) \otimes_{C_0(M)} L^2(M,\bF)$. 
The operator $\pS(\cdot)\otimes1$ is well-defined on $\Dom\pS(\cdot) \otimes_{C_0(M)} L^2(M,\bF) \subset L^2(M,E\otimes\bF)$, and is denoted simply by $\pS(\cdot)$ as well. By \cite[Proposition 9.10]{Lance95}, $\pS(\cdot)$ is regular self-adjoint on $L^2(M,E\otimes\bF)$. 

The operator $1\otimes\D$ is not well-defined on $L^2(M,E\otimes\bF)$. Instead, using the canonical isomorphism $L^2(M,E\otimes\bF) \simeq E \otimes L^2(M,\bF)$, we consider the operator $1\otimes\D$ on $E \otimes L^2(M,\bF)$ with domain $E\otimes\Dom\D$. 
Alternatively, we can extend the exterior derivative on $C_0^1(M)$ to an operator 
\[
d \colon C_0^1(M,E) \xrightarrow{\simeq} E\otimes C_0^1(M) \xrightarrow{1\otimes d} E\otimes\Gamma_0(T^*M) \xrightarrow{\simeq} \Gamma_0(E\otimes T^*M) .
\]
Denoting by $\sigma$ the principal symbol of $\D$, we can define an operator $1\otimes_d\D$ on $C_0(M,E) \otimes_{C_0(M)} L^2(M,\bF)$ by setting
$$
(1\otimes_d\D)(\xi\otimes\psi) := \xi\otimes\D\psi + (1\otimes\sigma)(d\xi)\psi .
$$
Under the isomorphism $C_0(M,E) \otimes_{C_0(M)} L^2(M,\bF) \simeq E \otimes L^2(M,\bF)$, the operator $1\otimes\D$ on $E \otimes L^2(M,\bF)$ agrees with $1\otimes_d\D$ on $C_0(M,E) \otimes_{C_0(M)} L^2(M,\bF)$. We will denote this operator on $L^2(M,E\otimes\bF)$ simply as $\D$. The operator $\D$ is regular self-adjoint on $L^2(M,E\otimes\bF)$ (see also \cite[Theorem 5.4]{KL13}). 

\begin{defn}
Given $M$, $\D$, and $\pS(\cdot)$ satisfying the \ref{ass:standing}, we define the (generalised) \emph{Dirac-Schr\"odinger operator} 
\[
\D_\pS := \D - i \pS(\cdot) 
\]
on the initial domain $C_c^1(M,W) \otimes_{C_0^1(M)} \Dom\D$. We denote by $\Dom\D_\pS$ the closure of this domain in the graph norm of $\D_\pS$. 
We also define the \emph{product operator}
\[
\til\D_\pS := \pS(\cdot)\times\D := \mattwo{0}{\D+i\pS(\cdot)}{\D-i\pS(\cdot)}{0} .
\]
\end{defn}
We note that, despite our use of the term `Dirac-Schr\"odinger' operator, we do not assume that the operator $\D$ is of Dirac-type (although a Dirac-type operator is of course a typical example). 

We introduce some further notation. 
For any $x\in M$, 
we write $\pS^x(\cdot)$ for the operator corresponding to the constant family $\pS^x(y) := \pS(x)$ (for all $y\in M$). 
We can then consider the product operator $\til\D_\pS^x := \pS^x(\cdot)\times\D$ defined as above. 
Lastly, we define on $L^2(M,E\otimes\bF)^{\oplus2}$ the operators
\begin{align*}
\til\D &:= \mattwo{0}{\D}{\D}{0} , & 
\til\pS(\cdot) &:= \mattwo{0}{+i\pS(\cdot)}{-i\pS(\cdot)}{0} , & \til\pS^x(\cdot) &:= \mattwo{0}{+i\pS^x(\cdot)}{-i\pS^x(\cdot)}{0} , 
\end{align*}
on the initial domain $\big(C_c^1(M,W) \otimes_{C_0^1(M)} \Dom\D\big)^{\oplus2}$. 

\begin{remark}
The product operator $\til\D_\pS$ is given by the standard formula for the (odd) unbounded Kasparov product of $\pS(\cdot)$ with $\D$ \cite{Mes14,KL13,BMS16,MR16}. In these references, the proof that this formula indeed represents the Kasparov product, relies on the condition that the commutator $[\D,\pS(\cdot)]$ is well-behaved. However, since we only assume that $\pS(\cdot)$ is continuous (but not necessarily differentiable), we have no direct control on this commutator. In \cref{sec:Kasp_prod} we will show that the operator $\til\D_\pS$ nevertheless represents the Kasparov product. 
\end{remark}

\subsubsection{Regularity and self-adjointness}

Our first task is to prove that the product operator $\til\D_\pS := \pS(\cdot)\times\D$ is regular self-adjoint. 
Our proof is a generalisation of the methods used in \cite[\S2]{AW11} for the case $M=[0,1]\subset\R$. 
We start with the following well-known lemma (for more general statements, see e.g.\ \cite[Theorem 6.1.8]{Mes14} or \cite[Theorem 7.10]{KL12}). 

\begin{lem}
\label{lem:sum_sa}
For $j=1,2$, let $T_j$ be a regular self-adjoint operator on the Hilbert $B_j$-module $E_j$. Then $T_1\otimes1 \pm i\otimes T_2$ is regular on the tensor product $E_1\otimes E_2$, and $(T_1\otimes1 \pm i\otimes T_2)^* = T_1\otimes1 \mp i\otimes T_2$. 
\end{lem}

\begin{lem}
\label{lem:sum_sa_small_var}
Suppose there exist $x_0\in M$ and $a<1$ such that $\big\| \big( S(x)-S(x_0) \big) \big(S(x_0)\pm i\big)^{-1} \big\| \leq a$ for all $x\in M$. Then $\til\D_\pS$ is regular self-adjoint on the domain $\Dom\til\D_\pS = \Dom\til\pS(\cdot) \cap \Dom\til\D$. 
\end{lem}
\begin{proof}
The operator $\pS^{x_0}(\cdot)$ on $L^2(M,E\otimes\bF)$ can be viewed as the operator $\pS(x_0)\otimes1$ on $E\otimes L^2(M,\bF)$. 
Applying \cref{lem:sum_sa} to $\pS(x_0)$ and $\D$, it follows that $\til\D_\pS^{x_0}(\cdot)$ is self-adjoint on the domain $(\Dom\pS(x_0)\otimes\Dom\D)^{\oplus2} = \Dom\til\pS^{x_0}(\cdot)\cap\Dom\til\D$. 
Since $\til\pS^{x_0}(\cdot)$ and $\til\D$ anti-commute, we have $\la\til\D_\pS^{x_0}\psi|\til\D_\pS^{x_0}\psi\ra = \la\til\pS^{x_0}(\cdot)\psi|\til\pS^{x_0}(\cdot)\psi\ra + \la\til\D\psi|\til\D\psi\ra$ for any $\psi\in\Dom\til\pS^{x_0}(\cdot)\cap\Dom\til\D$. Therefore
\begin{align*}
\big\| \big(\til\pS^{x_0}(\cdot)\pm i\big) \psi \big\|^2 &= \big\| \la\psi|\psi\ra + \big\la\til\pS^{x_0}(\cdot)\psi \mvert \til\pS^{x_0}(\cdot)\psi\big\ra \big\| 
\leq \big\| \la\psi|\psi\ra + \big\la\til\D_\pS^{x_0}\psi \mvert \til\D_\pS^{x_0}\psi\big\ra \big\| = \|\psi\|_{\til\D_\pS^{x_0}}^2 .
\end{align*}
It follows from \cref{lem:rel_bdd} that $(\til\pS^{x_0}(\cdot)\pm i)(\til\D_\pS^{x_0}\pm i)^{-1}$ is bounded with norm at most $1$.
Therefore we also have 
\begin{align*}
\|(\til\D_\pS-\til\D_\pS^{x_0}) (\til\D_\pS^{x_0}\pm i)^{-1}\| 
&\leq \|(\til\pS(\cdot)-\til\pS^{x_0}(\cdot)) (\til\pS^{x_0}(\cdot)\pm i)^{-1}\| \; \|(\til\pS^{x_0}(\cdot)\pm i) (\til\D_\pS^{x_0}\pm i)^{-1}\| \nonumber\\
&\leq \sup_{x\in M} \|(\pS(x)-\pS(x_0)) (\pS(x_0)\pm i)^{-1}\| \leq a < 1 . 
\end{align*}
By the Kato-Rellich theorem, it follows that $\til\D_\pS$ is regular self-adjoint on the domain $\Dom\til\D_\pS^{x_0} = \Dom\til\pS^{x_0}(\cdot)\cap\Dom\til\D$. 
Finally, we know from \cref{lem:uegn} that $\Dom\til\pS^{x_0}(\cdot) = \Dom\til\pS(\cdot)$. 
\end{proof}

We will make use of the following generalisation of Tietze's extension theorem. 
\begin{thm}[{\cite[Theorem 4.1]{Dug51}}]
\label{thm:Dug}
Let $X$ be an arbitrary metric space, $A$ a closed subset of $X$, $L$ a locally convex linear space, and $f\colon A\to L$ a continuous map. Then there exists a continuous extension $F\colon X\to L$ of $f$ such that $F(X)$ is contained in the convex hull of $f(A)$.
\end{thm}

\begin{lem}
\label{lem:extension}
Let $a\in[0,1)$, and consider a point $x\in M$ with an open neighbourhood $U\subset M$. 
\begin{enumerate}
\item If $\sup_{y\in\bar U} \|(\pS(y)-\pS(x)) (\pS(x)\pm i)^{-1}\| \leq a$, then there exists a family of regular self-adjoint operators $\{\pS^U(y)\}_{y\in M}$ satisfying \ref{ass:A1} and \ref{ass:A2} such that $\pS^U(y) = \pS(y)$ for all $y\in\bar U$, and 
$\sup_{y\in M} \|(\pS^U(y)-\pS(x)) (\pS(x)\pm i)^{-1}\| \leq a$. 
\item If $S(x)$ is invertible and $\sup_{y\in\bar U} \|(\pS(y)-\pS(x)) \pS(x)^{-1}\| \leq a$, then there is a family of regular self-adjoint operators $\{\pS^U(y)\}_{y\in M}$ satisfying \ref{ass:A1} and \ref{ass:A2} such that $\pS^U(y) = \pS(y)$ for all $y\in\bar U$, 
$\sup_{y\in M} \|(\pS^U(y)-\pS(x)) \pS(x)^{-1}\| \leq a$, and 
$\sup_{y\in M} \|\pS^U(y)^{-1}\| \leq \frac1{1-a} \|\pS(x)^{-1}\|$. 
\end{enumerate}
\end{lem}
\begin{proof}
We view $W_x := \Dom\pS(x)$ as a Hilbert $B$-module equipped with the graph norm of $\pS(x)$. 
Let $\Hom_B^s(W_x,E)$ denote the (real) Banach space of symmetric operators $T$ on $E$ with $\Dom T = W_x$ (equipped with the operator norm of maps $W_x\to E$). 
We then have a metric space $X=M$, a closed subset $A=\bar U$, a Banach space $L=\Hom_B^s(W_x,E)$, and a continuous map $\pS\colon A\to L$ whose image is contained in the ball of radius $a$ around $\pS(x)$. 
By \cref{thm:Dug} there then exists a continuous extension $\pS^U\colon M\to\Hom_B^s(W_x,E)$ whose image is contained in the same ball, i.e.\ such that $\sup_{y\in M} \|(\pS^U(y)-\pS(x)) (\pS(x)\pm i)^{-1}\| \leq a$. 
It follows from the Kato-Rellich theorem that $\pS^U(y)$ is regular self-adjoint (on the domain $W_x$) for all $y\in M$. This proves the first statement. 

The second statement is proven similarly, by equipping $W_x$ with the equivalent norm $\|\psi\|_{W_x} := \|\pS(x)\psi\|$ (in this case, we note that the Kato-Rellich theorem still applies, because $\|\pS(x) (\pS(x)\pm i)^{-1}\| \leq 1$). Furthermore, in this case, using the inequality $\|(\pS^U(y)-\pS(x)) \pS(x)^{-1}\| \leq a$, a Neumann series argument shows that $\|\pS^U(y)^{-1}\| \leq \frac1{1-a} \|\pS(x)^{-1}\|$ (for any $y\in M$). 
\end{proof}

\begin{prop}
\label{prop:sum_sa}
The product operator $\til\D_\pS := \pS(\cdot)\times\D$ is regular self-adjoint on $\Dom\til\D_\pS = \Dom\til\pS(\cdot) \cap \Dom\til\D$. 
\end{prop}
\begin{proof}
By assumption \ref{ass:A4}, there exists a finite open cover $\{V_j\}_{j=1}^l$ of $M\backslash K$ with points $x_j\in V_j$ and positive numbers $a_j<1$ such that $\big\| \big( \pS(x)-\pS(x_j) \big) \pS(x_j)^{-1} \big\| \leq a_j$ for all $x\in V_j$. 
For $x\in K$, we know by continuity that there exists a precompact open neighbourhood $U_x$ of $x$ such that $\sup_{y\in U_x} \|(\pS(y)-\pS(x)) (\pS(x)\pm i)^{-1}\| \leq \frac12$. 
Since $K$ is compact, there exist finitely many points $x_{l+k}\in K$ such that the open sets $V_{l+k} := U_{x_{l+k}}$ cover $K$. Setting $a_{l+k} := \frac12$, we therefore have a finite open cover $\{V_j\}$ of $M$ and numbers $a_j<1$ such that $\big\| \big( \pS(x)-\pS(x_j) \big) \big(\pS(x_j)\pm i\big)^{-1} \big\| \leq a_j$ for any $x\in V_j$. 

By \cref{lem:extension}, there exists for each $j$ a family of self-adjoint operators $\{\pS^{V_j}(x)\}_{x\in M}$ such that $\pS^{V_j}(x) = \pS(x)$ for all $x\in\bar{V_j}$, and $\sup_{x\in M} \|(\pS^{V_j}(x)-\pS(x_j)) (\pS(x_j)\pm i)^{-1}\| \leq a_j$. 
By \cref{lem:sum_sa_small_var}, the corresponding product operators $\til\D_\pS^{V_j}$ are regular self-adjoint on $\Dom\til\D_\pS^{V_j} = \Dom\til\pS(\cdot) \cap \Dom\til\D$. 
Let $\{\chi_j^2\}$ be a smooth partition of unity subordinate to the open cover $\{V_j\}$ of $M$. 
For some $\lambda\in\R\backslash\{0\}$, consider the operator
\[
R_\pm(\lambda) := \sum_j \chi_j \big( \til\D_\pS^{V_j} \pm i\lambda \big)^{-1} \chi_j .
\]
For any $x_0\in M$, we have $\Ran R_\pm(\lambda) \subset \Dom\til\pS(\cdot) \cap \Dom\til\D = \Dom\til\D_\pS^{x_0} \subset \Dom\til\D_\pS$, where the second inclusion follows because
\begin{align*}
\big\| (\til\D_\pS-\til\D_\pS^{x_0}) (\til\D_\pS^{x_0}\pm i)^{-1} \big\| 
&\leq \big\| (\til\pS(\cdot)-\til\pS^{x_0}(\cdot)) (\til\pS^{x_0}(\cdot)\pm i)^{-1} \big\| \; \big\| (\til\pS^{x_0}(\cdot)\pm i) (\til\D_\pS^{x_0}\pm i)^{-1} \big\| \\
&\leq \big\| (\til\pS(\cdot)-\til\pS^{x_0}(\cdot)) (\til\pS^{x_0}(\cdot)\pm i)^{-1} \big\|
\end{align*}
is bounded by \cref{lem:uegn}. 
For any $\psi\in L^2(M,E\otimes\bF)^{\oplus2}$ we have
\[
\big(\til\D_\pS\pm i\lambda\big) R_\pm(\lambda) \psi = \psi + \sum_j [\til\D,\chi_j] \big(\til\D_\pS^{V_j}\pm i\lambda\big)^{-1} \chi_j \psi .
\]
We can pick $\lambda$ sufficiently large, such that the norm of $\mK_\pm(\lambda) := \sum_j [\til\D,\chi_j] \big(\til\D_\pS^{V_j}\pm i\lambda\big)^{-1} \chi_j$ is less than one. Then $1+\mK_\pm(\lambda)$ is invertible, and $R_\pm(\lambda)(1+\mK_\pm(\lambda))^{-1}$ is a right inverse of $\til\D_\pS\pm i\lambda$. Similarly, we can also obtain a left inverse, which proves that $\til\D_\pS$ is regular self-adjoint on the domain $\Dom\til\D_\pS = \Dom\til\pS(\cdot) \cap \Dom\til\D$. 
\end{proof}

\section{Index theory}
\label{sec:index}

\subsection{The Fredholm property}

In this section, we will show that the operator $\til\D_\pS$ is Fredholm, and therefore $\Index(\D-i\pS(\cdot))$ is well-defined. 
The Fredholm property will be a consequence of the local compactness of the resolvents and the `invertibility near infinity'. 

\begin{prop}[{\cite[Theorem 6.7]{KL13}}]
\label{prop:cpt_res}
Let $\phi\in C_0(M)$. 
Then $\phi(\til\D_\pS\pm i)^{-1}$ is a compact operator on $L^2(M,E\otimes\bF)^{\oplus2}$. 
Moreover, if $(\pS(\cdot)\pm i)^{-1}$ is compact on $C_0(M,E)$, then $(\til\D_\pS\pm i)^{-1}$ is also compact. 
\end{prop}
\begin{proof}
The statement follows by observing that the argument in \cite[Theorem 6.7]{KL13} remains valid without assuming any differentiability for the family $\{\pS(x)\}_{x\in M}$. 
Moreover, if $(\pS(\cdot)\pm i)^{-1}$ is in fact compact, then we can repeat the argument with $\phi=1$ to conclude that also $(\til\D_\pS\pm i)^{-1}$ is compact. 
\end{proof}

\begin{lem}
\label{lem:inv_glob}
Suppose that $\pS(x)$ is uniformly invertible for all $x\in M$, and there exist $x_0\in M$ and $a<1$ such that $\big\| \big( S(x)-S(x_0) \big) S(x_0)^{-1} \big\| \leq a$ for all $x\in M$. Then $\til\D_\pS$ is invertible. 
\end{lem}
\begin{proof}
Since $\pS^{x_0}(\cdot)$ is invertible, the positive operator $\pS^{x_0}(\cdot)^2 +\D^2 = \big(\D-i\pS^{x_0}(\cdot)\big)\big(\D+i\pS^{x_0}(\cdot)\big)$ is invertible and therefore $\D\pm i\pS^{x_0}(\cdot)$ is invertible. 
Thus we know that $\til\D_\pS^{x_0}$ is invertible. 

Since $\til\pS^{x_0}(\cdot)$ and $\til\D$ anti-commute, we have $\la\til\D_\pS^{x_0}\psi|\til\D_\pS^{x_0}\psi\ra = \la\til\pS^{x_0}(\cdot)\psi|\til\pS^{x_0}(\cdot)\psi\ra + \la\til\D\psi|\til\D\psi\ra$ for any $\psi\in\Dom\til\pS^{x_0}(\cdot)\cap\Dom\til\D$, which implies that $\|\til\pS^{x_0}(\cdot)\psi\| \leq \|\til\D_\pS^{x_0}\psi\|$. 
It follows that $\|\til\pS^{x_0}(\cdot) (\til\D_\pS^{x_0})^{-1}\| \leq 1$, and we obtain the inequality
\begin{align}
\|(\til\D_\pS-\til\D_\pS^{x_0}) (\til\D_\pS^{x_0})^{-1}\| &\leq \|(\til\pS(\cdot)-\til\pS^{x_0}(\cdot)) \til\pS^{x_0}(\cdot)^{-1}\| \; \|\til\pS^{x_0}(\cdot) (\til\D_\pS^{x_0})^{-1}\| \nonumber\\
&\leq \sup_{x\in M} \|(\pS(x)-\pS(x_0)) \pS(x_0)^{-1}\| \leq a < 1 . 
\end{align}
A Neumann series argument then shows that $\til\D_\pS$ is also invertible, and that its inverse is given by
\[
\big( \til\D_\pS \big)^{-1} = \big( \til\D_\pS^{x_0} \big)^{-1} \sum_{n=0}^\infty \Big( \big( \til\D_\pS^{x_0}-\til\D_\pS \big) \big( \til\D_\pS^{x_0} \big)^{-1} \Big)^n .
\qedhere
\]
\end{proof}

\begin{thm}
\label{thm:sum_Fredholm}
Let $M$, $\D$, and $\pS(\cdot)$ be as in the \ref{ass:standing}. 
Then the product operator $\til\D_\pS$ is regular self-adjoint and Fredholm. 
Hence it defines a class $[\til\D_\pS] \in \KK^0(\C,B) \simeq \K_0(B)$ given by the index of the Dirac-Schr\"odinger operator $\D-i\pS(\cdot)$. 
\end{thm}
\begin{proof}
We have seen in \cref{prop:sum_sa} that $\til\D_\pS$ is regular self-adjoint. 
Let $V_0$ be a precompact open neighbourhood of $K$, which combined with the open subsets $V_j$ from assumption \ref{ass:A4} provides a finite open cover of $M$. 
Let $\{\chi_j^2\}$ be a partition of unity subordinate to $\{V_j\}$, such that $\chi_j(x)=0$ for $x\in K$ and for all $j\neq0$. 
By \cref{lem:inv_glob,lem:extension} we have invertible operators $\til\D_\pS^{V_j}$ for $j\neq0$. 
Define the operator 
$$
Q := \chi_0 (\til\D_\pS-i)^{-1} \chi_0 + \sum_{j\neq0} \chi_j \big(\til\D_\pS^{V_j}\big)^{-1} \chi_j .
$$
We then calculate that 
\begin{align*}
\til\D_\pS Q - 1 &= \big[\til\D_\pS,\chi_0\big] (\til\D_\pS-i)^{-1} \chi_0 + \chi_0^2 + i \chi_0 (\til\D_\pS-i)^{-1} \chi_0 + \sum_{j\neq0} \big[\til\D_\pS^{V_j},\chi_j\big] \big(\til\D_\pS^{V_j}\big)^{-1} \chi_j + \sum_{j\neq0} \chi_j^2 - 1 \\
&= \big[\til\D,\chi_0\big] (\til\D_\pS-i)^{-1} \chi_0 + i \chi_0 (\til\D_\pS-i)^{-1} \chi_0 + \sum_{j\neq0} \big[\til\D,\chi_j\big] \big(\til\D_\pS^{V_j}\big)^{-1} \chi_j . 
\end{align*}
The operators $[\til\D,\chi_j]$ are smooth and compactly supported, 
and therefore bounded. Since $(\til\D_\pS^{V_j}-i) \big(\til\D_\pS^{V_j}\big)^{-1} \chi_j$ is also bounded (for $j\neq0$), it follows from \cref{prop:cpt_res} that $\til\D_\pS Q - 1$ is compact. 
Hence $Q$ is a right parametrix for $\til\D_\pS$. A similar calculation shows that $Q$ is also a left parametrix, and therefore $\til\D_\pS$ is Fredholm. 
\end{proof}

\begin{lem}
\label{lem:Fredholm_rescaling}
For any $\lambda>0$, we have $\Index\big(\D-i\lambda\pS(\cdot)\big) = \Index\big(\D-i\pS(\cdot)\big)$. 
\end{lem}
\begin{proof}
Observing that the rescaled $\lambda\pS(\cdot)$ still satisfies the assumptions \ref{ass:A1}-\ref{ass:A4}, we know from \cref{thm:sum_Fredholm} that $\D-i\lambda\pS(\cdot)$ is Fredholm for any $\lambda>0$. Similarly, the family $\{\pS^\bullet(\cdot)\}_{x\in M}$ on $C([0,1],E)$, given by $\pS^t(x) := (1-t+t\lambda)\pS(x)$, satisfies the assumptions \ref{ass:A1}-\ref{ass:A4}, so $\D-i\pS^\bullet(\cdot)$ provides a homotopy between $\D-i\lambda\pS(\cdot)$ and $\D-i\pS(\cdot)$. 
\end{proof}

\begin{remark}
\begin{enumerate}
\item We mention a special case in which the index of the Dirac-Schr\"odinger operator $\D-i\pS(\cdot)$ vanishes automatically.
Suppose there exists a self-adjoint unitary endomorphism $\Gamma\in\Gamma^\infty(M,\End\bF)$ such that $\Gamma\cdot\Dom\D\subset\Dom\D$, and $\Gamma\D=-\D\Gamma$. 
Then
\begin{align*}
\Index\big( \D-i\pS(\cdot) \big) &= \Index\big( \Gamma \big( \D-i\pS(\cdot) \big) \Gamma \big) = \Index\big( \D+i\pS(\cdot) \big) = - \Index\big( \D-i\pS(\cdot) \big) ,
\end{align*}
hence $\Index\big( \D-i\pS(\cdot) \big) = 0$. 
Such $\Gamma$ exists, for instance, if $\D$ is of Dirac-type and $M$ is even-dimensional. 

\item Consider the \emph{classical} case of a Dirac-Schr\"odinger operator $\D-iV$ (also called a \emph{Callias-type} operator), where $\D$ is a Dirac-type operator, and the potential $V$ is a self-adjoint endomorphism on some auxiliary vector bundle (of \emph{finite rank}). Under suitable assumptions on the potential $V$, one can then prove that $\D-i\lambda V$ is Fredholm for \emph{sufficiently large} $\lambda\in(0,\infty)$ (see, for instance, \cite{BM92,Ang93a}). 

In our case, $\pS(\cdot)$ plays the role of the potential. 
As observed in \cref{lem:Fredholm_rescaling}, $\D-i\lambda\pS(\cdot)$ is Fredholm for \emph{any} $\lambda\in(0,\infty)$.
However, we stress here that this does not provide a generalisation of the classical result. Indeed, our theorem applies to a `potential' $\pS(\cdot)$ acting on 
the Hilbert module $C_0(M,E)$. 
Restricting to the case of finite-rank bundles, this means our theorem only applies to \emph{trivial} bundles. 
In the case of non-trivial bundles, there are examples in which the Fredholm property of $\D-i\lambda V$ fails to hold for some $\lambda\in(0,\infty)$ (see, e.g., \cite[\S4]{BM92}). 

\item In the classical case of a Dirac-Schr\"odinger operator $\D-iV$, the index vanishes whenever the manifold is compact. Indeed, in this case $V$ is bounded and $\D$ has compact resolvents. Therefore $\D-iV$ is a relatively compact perturbation of $\D$ (which is self-adjoint), so that $\Index(\D-iV)=\Index(\D)=0$. 

In our setup, the Hilbert $B$-module $E$ is in general not finitely generated projective (i.e., in the case $B=\C$, the Hilbert space $E=\mH$ is in general not finite-dimensional). 
Therefore, although $\D$ has compact resolvents on $L^2(M,\bF)$, the operator $\D\otimes1$ on $L^2(M,E\otimes\bF)$ in general does \emph{not} have compact resolvents. Hence the index may be non-zero on compact manifolds. 
Consider for instance the simple case of a circle $M=S^1$ (with $\D=-i\partial_x$). Then the index of $\partial_x+\pS(\cdot)$ is equal to the spectral flow of $\{\pS(x)\}_{x\in S^1}$ (as we will discuss at the start of \cref{sec:Kasp_prod}), and this spectral flow could certainly be non-zero. 
\end{enumerate}
\end{remark}

\subsection{A relative index theorem}

Before we describe the relative index theorem, we will first prove an auxiliary lemma, stating that the class of an unbounded Fredholm operator vanishes if there exists a suitable symmetry. A similar statement is given in \cite[Lemma 1.15]{Bun95} in terms of bounded Kasparov modules. 

\begin{lem}[cf.\ {\cite[Lemma 1.15]{Bun95}}]
\label{lem:cliff_triv}
Let $\D$ be an odd regular self-adjoint Fredholm operator on a $\Z_2$-graded Hilbert $B$-module $E = E^0\oplus E^1$. 
Suppose there exists an odd self-adjoint unitary endomorphism $e\in\End_B(E)$ such that $e\cdot\Dom\D\subset\Dom\D$, and the graded commutator $[\D,e]_\pm$ is bounded and relatively compact (i.e.\ $[\D,e]_\pm(\D\pm i)^{-1}$ is compact). 
Then $[\D] = 0 \in \KK^0(\C,B)$. 
\end{lem}
\begin{proof}
Since $e[\D,e]_\pm$ is bounded and symmetric, we know that $\D_t := \D - \frac t2 e[\D,e]_\pm$ is regular self-adjoint for each $t\in[0,1]$. Furthermore, since $[\D,e]_\pm$ is relatively compact, we also know that $\D_t$ is Fredholm, and any parametrix for $\D$ is also a parametrix for $\D_t$. 
Then the operator $\D_\bullet$ on $C([0,1],E)$ is also regular self-adjoint and Fredholm, 
so it provides a homotopy between $\D_0 = \D$ and $\D_1 = \frac12(\D-e\D e)$.
Since $e$ is an odd unitary, we can identify $E^0$ with $E^1$ using $e$, so that $E\simeq E^0\oplus E^0$. Then $e$, $\D$, and $\D_1$ take the form
\begin{align*}
e &= \mattwo{0}{1}{1}{0} , & 
\D &= \mattwo{0}{\D_-}{\D_+}{0} , & 
\D_1 &= \frac12 \mattwo{0}{\D_--\D_+}{\D_+-\D_-}{0} .
\end{align*}
The class of $\D_1$ is given by $\Index((\D_1)_+) \in \K_0(B)$. Since $(\D_1)_+ = \frac12(\D_+-\D_-)$ is skew-adjoint, this index vanishes. Hence $[\D] = [\D_1] = 0$. 
\end{proof}

For $j=1,2$, let $\bF^j\to M^j$, $\D^j$, and $\pS^j(\cdot)$ be as in the \ref{ass:standing}. 
We assume that the operators $\{\pS^j(x)\}_{x\in M^j}$ act on the same Hilbert $B$-module $E$. 
Suppose we have partitions $M^j = \bar U^j \cup_{N^j} \bar V^j$, where $N^j$ are smooth compact hypersurfaces. 
Let $C^j$ be precompact open tubular neighbourhoods of $N^j$, and assume that there exists an isometry $\phi\colon C^1\to C^2$ (with $\phi(N^1)=N^2$) covered by an isomorphism $\Phi\colon\bF^1|_{C^1} \to \bF^2|_{C^2}$, such that $\D^1|_{C^1} \Phi^* = \Phi^* \D^2|_{C^2}$ and $\pS^2(\phi(x)) = \pS^1(x)$ for all $x\in C^1$. 
We will identify $N^1$ and $N^2$ via $\phi$, and we simply write $N$. Define two new Riemannian manifolds 
\begin{align*}
M^3 &:= \bar U^1 \cup_N \bar V^2 , & 
M^4 &:= \bar U^2 \cup_N \bar V^1 .
\end{align*}
Moreover, we glue the bundles using $\Phi$ to obtain hermitian vector bundles $\bF^3\to M^3$ and $\bF^4\to M^4$. For $j=3,4$, we then obtain corresponding operators $\D^j$ and $\pS^j(\cdot)$ satisfying the \ref{ass:standing}. 

For $j=1,\ldots,4$, we consider the product operators $\til\D_\pS^j := \pS^j(\cdot)\times\D^j$. By \cref{thm:sum_Fredholm}, these operators are regular self-adjoint and Fredholm. Under the standard isomorphism $\KK^0(\C,B) \simeq \K_0(B)$, the class $[\til\D_\pS^j]\in\KK^0(\C,B)$ is equal to $\Index(\D^j-i\pS^j(\cdot))\in\K_0(B)$. 

\begin{thm}[Relative index theorem]
\label{thm:rel_index}
We have $[\til\D_\pS^1]+[\til\D_\pS^2]=[\til\D_\pS^3]+[\til\D_\pS^4]\in\KK^0(\C,B)$, and hence 
\[
\Index(\D^1-i\pS^1(\cdot)) + \Index(\D^2-i\pS^2(\cdot)) = \Index(\D^3-i\pS^3(\cdot)) + \Index(\D^4-i\pS^4(\cdot)) . 
\]
\end{thm}
\begin{proof}
We roughly follow Bunke's proof of the $\K$-theoretic relative index theorem \cite[Theorem 1.14]{Bun95}, except that we work with unbounded operators.
For $j=1,\ldots,4$, we write $\til E^j := L^2(M^j,E\otimes\bF^j)^{\oplus2}$. Each $\til E^j$ is $\Z_2$-graded, with the grading operator $\Gamma^j$ given by $1\oplus(-1)$. 
We write $(\til E^j)^\op = \til E^j$ for the same Hilbert module with the opposite grading $(\Gamma^j)^\op = (-1)\oplus1$. 
We define the Hilbert module $\til E := \til E^1\oplus \til E^2\oplus (\til E^3)^\op\oplus (\til E^4)^\op$, which is $\Z_2$-graded with the grading operator $\Gamma = \Gamma^1\oplus \Gamma^2\oplus (\Gamma^3)^\op\oplus (\Gamma^4)^\op$. We consider the operator $\til\D_\pS := \til\D_\pS^1\oplus\til\D_\pS^2\oplus\til\D_\pS^3\oplus\til\D_\pS^4$ on $\til E$. 
Since $(\C,(\til E^j)^\op,\til\D_\pS^j)$ represents the class $-[\til\D_\pS^j]$, we need to show that $[\til\D_\pS] = [\til\D_\pS^1] + [\til\D_\pS^2] - [\til\D_\pS^3] - [\til\D_\pS^4] = 0\in\KK^0(\C,B)$. 

For $j=1,2$, we choose smooth functions $\chi_{U^j}$ and $\chi_{V^j}$ such that 
\begin{gather*}
\begin{aligned}
\supp \chi_{U^j} &\subset U^j \cup C^j , \qquad&\qquad
\supp \chi_{V^j} &\subset V^j \cup C^j , 
\end{aligned}\\
\begin{aligned}
\chi_{U^j}^2 + \chi_{V^j}^2 &= 1 , \quad&\quad
\chi_{U^1}|_{C^1} &= \phi^*\chi_{U^2}|_{C^2} , \quad&\quad
\chi_{V^1}|_{C^1} &= \phi^*\chi_{V^2}|_{C^2} .
\end{aligned}
\end{gather*}
Note that $L^2(U^1\cup C^1,E\otimes\bF^1)$, which is a subspace of $L^2(M^1,E\otimes\bF^1)$, can also be viewed as a subspace of $L^2(M^3,E\otimes\bF^3)$, since $\bF^1|_{U^1\cup C^1} \simeq \bF^3|_{U^1\cup C^1}$. We write $\beta_{13}\colon L^2(U^1\cup C^1,E\otimes\bF^1) \to L^2(U^1\cup C^1,E\otimes\bF^3)$ for this identification. Similarly, we have identifications $\beta_{14}\colon L^2(V^1\cup C^1,E\otimes\bF^1) \to L^2(V^1\cup C^1,E\otimes\bF^4)$, $\beta_{23}\colon L^2(V^2\cup C^2,E\otimes\bF^2) \to L^2(V^2\cup C^2,E\otimes\bF^3)$, and $\beta_{24}\colon L^2(U^2\cup C^2,E\otimes\bF^2) \to L^2(U^2\cup C^2,E\otimes\bF^4)$. We define 
\begin{align*}
\alpha_{13} &:= \beta_{13} \circ \chi_{U^1} , & 
\alpha_{14} &:= \beta_{14} \circ \chi_{V^1} , & 
\alpha_{23} &:= \beta_{23} \circ \chi_{V^2} , & 
\alpha_{24} &:= \beta_{24} \circ \chi_{U^2} .
\end{align*}
Consider on $\til E$ the operator 
\[
X := \matfour{0&0&-\alpha_{13}^*&-\alpha_{14}^*}{0&0&-\alpha_{23}^*&\alpha_{24}^*}{\alpha_{13}&\alpha_{23}&0&0}{\alpha_{14}&-\alpha_{24}&0&0} .
\]
Then $X=-X^*$ and $X^2=-1$. 
We calculate 
\begin{align*}
[\til\D_\pS,X] &= \matfour{0&0&-\til\D_\pS^1\alpha_{13}^*+\alpha_{13}^*\til\D_\pS^3&-\til\D_\pS^1\alpha_{14}^*+\alpha_{14}^*\til\D_\pS^4}{0&0&-\til\D_\pS^2\alpha_{23}^*+\alpha_{23}^*\til\D_\pS^3&\til\D_\pS^2\alpha_{24}^*-\alpha_{24}^*\til\D_\pS^4}{\til\D_\pS^3\alpha_{13}-\alpha_{13}\til\D_\pS^1&\til\D_\pS^3\alpha_{23}-\alpha_{23}\til\D_\pS^2&0&0}{\til\D_\pS^4\alpha_{14}-\alpha_{14}\til\D_\pS^1&-\til\D_\pS^4\alpha_{24}+\alpha_{24}\til\D_\pS^2&0&0} \\
&= \matfour{0&0&-[\til\D^1,\chi_{U^1}]&-[\til\D^1,\chi_{V^1}]}{0&0&-[\til\D^2,\chi_{V^2}]&[\til\D^2,\chi_{U^2}]}{{[\til\D^1,\chi_{U^1}]}&[\til\D^2,\chi_{V^2}]&0&0}{{[\til\D^1,\chi_{V^1}]}&-[\til\D^2,\chi_{U^2}]&0&0} .
\end{align*}
Since all entries are smooth and compactly supported, $[\til\D_\pS,X]$ is bounded. Furthermore, from \cref{prop:cpt_res} we know that $[\til\D_\pS,X](\til\D_\pS\pm i)^{-1}$ is compact. 
Let $e := \Gamma X$. 
Since $\Gamma X = -X\Gamma$, we have $e=e^*$, $e^2=1$, and $[\til\D_\pS,e]_\pm = \til\D_\pS e + e \til\D_\pS = -\Gamma [\til\D_\pS,X]$ is bounded and relatively compact. 
It follows from \cref{lem:cliff_triv} that $[\til\D_\pS] = 0 \in \KK^0(\C,B)$. 
\end{proof}

\section{The Kasparov product}
\label{sec:Kasp_prod}

Consider a Dirac-Schr\"odinger operator $\D_\pS = \D -i \pS(\cdot)$ on a Riemannian manifold $M$. 
By \cref{prop:pS_Fred}, we know that the operator $\pS(\cdot)$ on the Hilbert $C_0(M,B)$-module $C_0(M,E)$ is Fredholm, and therefore we obtain a well-defined class $[\pS(\cdot)] \in \KK^1(\C,C_0(M,B))$ as defined in \cref{prop:Fred_KK}. 
Furthermore, since $\D$ is an essentially self-adjoint first-order differential operator, and since 
the ellipticity of $\D$ ensures that $\D$ also has locally compact resolvents \cite[Proposition 10.5.2]{Higson-Roe00}, we know that $(C_0^1(M), L^2(M,\bF), \D)$ is an odd spectral triple, which represents an odd $\K$-homology class $[\D] \in \KK^1(C_0(M),\C)$. 
Finally, the product operator $\til\D_\pS = \pS(\cdot)\times\D$ on the Hilbert module $L^2(M,E\otimes\bF)^{\oplus2}$ is Fredholm by \cref{thm:sum_Fredholm}, and its class $[\til\D_\pS] \in \KK^0(\C,B)$ is given by the index of $\D-i\pS(\cdot)$. 
In this section we will show that the class $[\til\D_\pS]$ equals the (internal) Kasparov product (over $C_0(M)$) of $[\pS(\cdot)]$ with $[\D]$. 
In order to prove this, we need to strengthen assumption \ref{ass:A4}, and we will replace it by:
\begin{enumerate}
\item[(A4')]
\customlabel{ass:A4'}{(A4')}
There exists a \emph{disjoint} finite open cover $\{V_j\}$ of $M\backslash K$ with points $x_j\in\bar V_j$ and positive numbers $a_j<1$ such that $\big\| \big( \pS(x)-\pS(x_j) \big) \pS(x_j)^{-1} \big\| \leq a_j$ for all $x\in V_j$. 
\end{enumerate}

The following proposition shows that, without loss of generality, we may assume $\pS(x)$ to be constant on each disjoint open subset $V_j$. 

\begin{prop}
\label{prop:pS_cst_ends}
The operator $\pS(\cdot)$ is homotopic to an operator $\pS'(\cdot)$ which is obtained from a family $\{\pS'(x)\}_{x\in M}$ satisfying the assumptions \ref{ass:A1}-\ref{ass:A4'} with a compact subset $K'$ and open subsets $V_j' \subset M\backslash K'$, such that $\pS'(x) = \pS(x_j)$ for all $x\in V_j'$. 
\end{prop}
\begin{proof}
Let $U$ be a precompact open neighbourhood of $K$ with smooth compact boundary $\partial U$, and let $C \simeq \partial U \times (-1,1)$ be a tubular neighbourhood of $\partial U$. We have $\partial U = \bigsqcup_j N_j$ with $N_j := V_j\cap\partial U$. 
We define the compact subset $K' := U \cup \bar C$ and the open subsets $V_j' := V_j \backslash K'$. 
We consider the homotopy given by
\[
H^t(x) := \begin{cases}
          \pS(x) , & x\in U , \\
          (1-tr) \pS(y,r) + tr \pS(x_j) , & x=(y,r)\in N_j\times[0,1] , \\
          (1-t) \pS(x) + t \pS(x_j) , & x \in V_j' .
          \end{cases}
\]
Defining $\pS'(x) := H^1(x)$, it is clear that $\pS'(x) = \pS(x_j)$ for all $x\in V_j'$. Furthermore, $H^t(x)$ depends continuously on $t$, and the assumptions on $\pS(x)$ then ensure that the family of operators $\{H^\bullet(x)\}_{x\in M}$ on the Hilbert $C([0,1],B)$-module $C([0,1],E)$ also satisfies the assumptions \ref{ass:A1}-\ref{ass:A4'}. From \cref{prop:pS_Fred} we then know that $H^\bullet(\cdot)$ is regular self-adjoint and Fredholm, so it indeed yields a homotopy between $\pS(\cdot)$ and $\pS'(\cdot)$. 
\end{proof}

Before we continue, let us first discuss the special case $M=\R$ and $B=\C$, with the operator $\D = -i\partial_x$. Then the equality $[\til\D_\pS] = [\pS(\cdot)] \otimes_{C_0(\R)} [-i\partial_x]$ follows more or less immediately from the `index = spectral flow' theorem combined with the description of spectral flow as a Kasparov product (given in \cref{prop:SF_Kasp_prod}). The equality between index and spectral flow was first thoroughly investigated by Robbin and Salamon \cite{RS95}. Here we will consider the generalisation from \cite{Wah07,AW11}. 

\begin{thm}[`Index = spectral flow', cf.\ {\cite[Theorem 2.1]{AW11}}]
\label{thm:index=sf}
Let $\{\pS(x)\}_{x\in\R}$ be a family of self-adjoint operators on a Hilbert space $\mH$ satisfying the \ref{ass:standing} (with $B=\C$, $E=\mH$, $M=\R$, and $\D=-i\partial_x$), and satisfying \ref{ass:A4'} with the open subsets $V_1 = (-\infty,a)$ and $V_2 = (b,\infty)$ for some $a<b\in\R$. 
Then the spectral flow of $\{\pS(x)\}_{x\in\R}$ is equal to the index of the operator $\partial_x+\pS(\cdot)$ on $L^2(\R,\mH)$.
\end{thm}
\begin{proof}
By \cref{lem:pS_loc_triv_fam}, there exist locally trivialising families for $\{\pS(x)\}_{x\in\R}$, so that the spectral flow is well-defined. 
We may assume (without loss of generality) that $a=0$ and $b=1$. 
Consider the family
\[
\pS'(x) := 
\begin{cases}
\pS(0) & x<0 , \\
\pS(x) & 0\leq x\leq 1 , \\
\pS(1) & x>1 .
\end{cases}
\]
By definition, $\SF\big(\{\pS(x)\}_{x\in\R}\big) = \SF\big(\{\pS(x)\}_{x\in[0,1]}\big) = \SF\big(\{\pS'(x)\}_{x\in\R}\big)$. 
Moreover, from \cref{prop:pS_cst_ends} we obtain a homotopy $H^\bullet(\cdot)$ between $\pS(\cdot)$ and $\pS'(\cdot)$, such that $H^\bullet(x)$ satisfies \ref{ass:A1}-\ref{ass:A4'}. By \cref{thm:sum_Fredholm}, $\partial_x + H^\bullet(\cdot)$ is Fredholm on $C([0,1],L^2(\R,E))$, and therefore it provides a homotopy between $\partial_x + \pS(\cdot)$ and $\partial_x + \pS'(\cdot)$. 
Hence we also have the equality $\Index\big(\partial_x+\pS(\cdot)\big) = \Index\big(\partial_x+\pS'(\cdot)\big)$. 
Finally, it has been shown in \cite[Theorem 2.1]{AW11} that the spectral flow $\SF\big(\{\pS'(x)\}_{x\in[0,1]}\big)$ is equal to $\Index\big(\partial_x+\pS'(\cdot)\big)$, which completes the proof. 
\end{proof}

\begin{coro}
Let $\{\pS(x)\}_{x\in\R}$ be as in \cref{thm:index=sf}. 
Then the index of $\partial_x+\pS(\cdot)$ on $L^2(\R,\mH)$ represents the internal Kasparov product of $[\pS(\cdot)] \in \KK^1(\C,C_0(\R))$ with $[-i\partial_x] \in \KK^1(C_0(\R),\C)$. 
\end{coro}
\begin{proof}
We have the equalities
\[
\Index(\partial_x+\pS(\cdot)) = \SF\{\pS(x)\}_{x\in\R} = [\pS(\cdot)] \otimes_{C_0(\R)} [-i\partial_x] ,
\]
where the first equality is from \cref{thm:index=sf}, and the second equality is from \cref{prop:SF_Kasp_prod}. 
\end{proof}

In the remainder of this section we aim to prove the equality 
\[
\Index(\D-i\pS(\cdot)) = [\pS(\cdot)] \otimes_{C_0(M)} [\D] ,
\]
on a manifold $M$ of arbitrary dimension, where the operators $\pS(\cdot)$ and $\D$ satisfy the \ref{ass:standing} as well as assumption \ref{ass:A4'}. 
This equality can be interpreted as a generalisation of the `index = spectral flow' theorem to higher-dimensional manifolds. 
We point out that a similar equality has already been obtained in \cite{KL13} under the assumption that the family $\{\pS(x)\}_{x\in M}$ is suitably differentiable. First, we will adapt the methods of \cite[\S8]{KL13} to obtain the equality $\Index(\D-i\pS(\cdot)) = [\pS(\cdot)] \otimes_{C_0(M)} [\D]$ for a differentiable family $\{\pS(x)\}_{x\in M}$. Subsequently, we will show that the equality remains valid without assuming any differentiability.

\subsection{A differentiable family}
\label{sec:Kasp_prod_diff}

In this subsection, following \cite[\S8]{KL13}, we consider the special case in which the family $\{\pS(x)\}_{x\in M}$ is not only continuous but in fact differentiable. 
While in \cite[\S8]{KL13} only families of operators on a Hilbert \emph{space} were considered, we will show that the results of \cite[\S8]{KL13} remain valid for families of operators acting on a Hilbert \emph{module} over a $C^*$-algebra $B$. 

\begin{defn}[{cf.\ \cite[\S8.3]{KL13} \& \cite[\S4.3]{vdDR16}}]
Let $M$ be a smooth manifold, and let $E_1$ and $E_2$ be countably generated Hilbert $B$-modules. 
A map $S(\cdot)\colon M\to\Hom_B(E_1,E_2)$, $x\mapsto S(x)$, is said to have a \emph{uniformly bounded weak derivative} if the map is weakly differentiable (i.e.\ the map $x\mapsto \la S(x)\xi,\eta\ra$ is differentiable for each $\xi\in E_1$ and $\eta\in E_2$), the weak derivative $dS(x)\colon E_1 \to E_2\otimes T_x^*(M)$ is bounded for all $x\in M$, and the supremum $\sup_{x\in M}\|dS(x)\|$ is finite.
\end{defn}

\begin{defn}[see {\cite[Assumption 7.1]{KL12}}]
\label{defn:almost_(anti-)commute}
Let $S$ and $T$ be regular self-adjoint operators on a Hilbert $B$-module $E$. We will say that $[S,T](S-i\mu)^{-1}$ is \emph{well-defined and bounded} if
\begin{enumerate}
\item there exists a submodule $\E\subset\Dom T$ which is a core for $T$;
\item for each $\xi\in\E$ and for all $\mu\in\R\backslash\{0\}$ we have the inclusions
\begin{align*}
(S-i\mu)^{-1} \xi &\in \Dom S\cap\Dom T\quad\mbox{and}\quad  T(S-i\mu)^{-1}\xi \in \Dom S ;
\end{align*}
\item the map $[S,T] (S-i\mu)^{-1} \colon \E \to E$ extends to a bounded operator 
in $\End_B(E)$ for all $\mu\in\R\backslash\{0\}$.
\end{enumerate}
\end{defn}

\begin{assumption}
\label{ass:diff}
Let $B$ be a (trivially graded) $\sigma$-unital $C^*$-algebra, and let $E$ be a countably generated Hilbert $B$-module. 
Let $M$ be a connected \emph{complete} Riemannian manifold, and let $\D$ be a symmetric elliptic first-order differential operator \emph{with bounded propagation speed} on a hermitian vector bundle $\bF\to M$.
Let $\{\pS(x)\}_{x\in M}$ be a family of regular self-adjoint operators on $E$ satisfying assumptions \ref{ass:A1}-\ref{ass:A3}. 
\end{assumption}

Let $\sigma$ be the principal symbol of $\D$. The boundedness of the propagation speed of $\D$ means that $\sup_{\xi\in T^*M,\,\|\xi\|=1}\|\sigma(\xi)\|$ is bounded. 
We point out that the completeness of $M$ and the bounded propagation speed of $\D$ ensure that $\D$ is essentially self-adjoint \cite[Proposition 10.2.11]{Higson-Roe00}. 

\begin{lem}
\label{lem:comm_loc_est}
Let $\D$ and $\pS(\cdot)$ be as in \cref{ass:diff}. 
Suppose that the graph norms of $\pS(x)$ are uniformly equivalent, and that $\pS\colon M\to\Hom_B(W,E)$ has a uniformly bounded weak derivative. 
Let $\phi\in C^1(M)$ be any differentiable function such that $\phi(x)\geq1$ for all $x\in M$, and $d\phi$ is bounded. 
Then $\big[\D,\phi\pS(\cdot)\big] \big(\phi\pS(\cdot)\pm i\big)^{-1}$ is well-defined and bounded. 
Moreover, we have the following pointwise norm-estimates:
\begin{align*}
\left\| \big[\D,\phi\pS(\cdot)\big](x) \big(\phi(x)\pS(x)\pm i\big)^{-1} \right\| 
&\leq \|\sigma(d\phi)(x)\| + \phi(x) \left\| \big[\D,\pS(\cdot)\big] \big(\pS(\cdot)\pm i\big)^{-1}\right\| , & x&\in M , \\
\left\| \big[\D,\phi\pS(\cdot)\big](x) \big(\phi(x)\pS(x)\pm i\big)^{-1} \right\| 
&\leq \|\sigma(d\phi)(x)\| + (1+c^{-1}) \left\| \big[\D,\pS(\cdot)\big] \big(\pS(\cdot)\pm i\big)^{-1}\right\| , & x&\in M\backslash K ,
\end{align*}
where $c>0$ is chosen such that $[-c,c]$ does not intersect the spectrum of $\pS(x)$ for any $x\in M\backslash K$. 
\end{lem}
\begin{proof}
The operator $\big[\D,\pS(\cdot)\big] \big(\pS(\cdot)\pm i\big)^{-1}$ is well-defined and bounded, which follows from the same argument as in the proofs of \cite[Lemma 8.5 \& Theorem 8.6]{KL13} (in fact, in \cite{KL13} only the case $B=\C$ was considered, but the same argument also works for an arbitrary $C^*$-algebra $B$).
We have
\[
[\D,\phi\pS(\cdot)] = \sigma(d\phi) \pS(\cdot) + \phi [\D,\pS(\cdot)] . 
\]
Hence for any $x\in M$ we obtain 
\begin{align*}
\left\| \big[\D,\phi\pS(\cdot)\big](x) \big(\phi(x)\pS(x)\pm i\big)^{-1} \right\| 
&\leq \|\sigma(d\phi)(x)\| \; \|\pS(x)(\pS(x)\pm i)^{-1}\| \; \|(\pS(x)\pm i)(\phi(x)\pS(x)\pm i)^{-1}\| \\
&\quad+ \phi(x) \left\| \big[\D,\pS(\cdot)\big] \big(\pS(\cdot)\pm i\big)^{-1}\right\| \; \left\| \big(\pS(x)\pm i\big) \big(\phi(x)\pS(x)\pm i\big)^{-1} \right\| .
\end{align*}
We can estimate
\[
\left\| \big(\pS(x)\pm i\big) \big(\phi(x)\pS(x)\pm i\big)^{-1} \right\| \leq \sup_{s\in\R} \left| \frac{s\pm i}{\phi(x)s\pm i} \right| = \max(1,\phi(x)^{-1}) = 1 .
\]
Thus for any $x\in M$ we obtain
\begin{align*}
\left\| \big[\D,\phi\pS(\cdot)\big](x) \big(\phi(x)\pS(x)\pm i\big)^{-1} \right\| 
\leq \|\sigma(d\phi)(x)\| \cdot 1 \cdot 1 + \phi(x) \left\| \big[\D,\pS(\cdot)\big] \big(\pS(\cdot)\pm i\big)^{-1}\right\| \cdot 1 ,
\end{align*}
which proves the first inequality. For $x\in M\backslash K$, we can use the invertibility of $\pS(x)$ to improve this inequality. Indeed, we can estimate
\begin{align*}
\left\| \big(\pS(x)\pm i\big) \big(\phi(x)\pS(x)\pm i\big)^{-1} \right\| 
\leq \phi(x)^{-1} \big\| (\pS(x)\pm i) \pS(x)^{-1} \big\| \; \big\| \pS(x) (\pS(x)\pm i\phi(x)^{-1})^{-1} \big\| \leq \phi(x)^{-1} (1+c^{-1}) .
\end{align*}
Hence for $x\in M\backslash K$ we obtain
\begin{align*}
\left\| \big[\D,\phi\pS(\cdot)\big](x) \big(\phi(x)\pS(x)\pm i\big)^{-1} \right\| 
\leq \|\sigma(d\phi)(x)\| \cdot 1 \cdot 1 + \phi(x) \left\| \big[\D,\pS(\cdot)\big] \big(\pS(\cdot)\pm i\big)^{-1}\right\| \phi(x)^{-1} (1+c^{-1}) ,
\end{align*}
which proves the second inequality. In particular, since this second inequality does not depend on the size of $\phi(x)$, we conclude that $\big\| \big[\D,\phi\pS(\cdot)\big] \big(\phi\pS(\cdot)\pm i\big)^{-1} \big\| = \sup_{x\in M} \big\| \big[\D,\phi\pS(\cdot)\big](x) \big(\phi(x)\pS(x)\pm i\big)^{-1} \big\|$ is bounded. 
\end{proof}

\begin{lem}
\label{lem:diff_Fred}
Let $\D$ and $\pS(\cdot)$ be as in \cref{ass:diff}. 
Suppose that the operator $\big[\D,\pS(\cdot)\big] \big(\pS(\cdot)\pm i\big)^{-1}$ is well-defined and bounded. 
Then there exists $\lambda_0\geq1$ such that 
the operator $\D_{\lambda\pS} = \D - i \lambda \pS(\cdot)$ is Fredholm, for any $\lambda\geq\lambda_0$. 
\end{lem}
\begin{proof}
From \cref{lem:comm_loc_est} we know that $\big[\D,\lambda\pS(\cdot)\big] \big(\lambda\pS(\cdot)\pm i\big)^{-1}$ is well-defined and bounded. It then follows from \cite[Theorem 7.10]{KL12} that $\til\D_{\lambda\pS}$ is regular self-adjoint on the domain $\Dom\til\pS(\cdot)\cap\Dom\til\D$. 

Write $\mu_x := 1 + \big\| \big[\D,\lambda\pS(\cdot)\big](x) \big(\lambda\pS(x)\pm i\big)^{-1} \big\|^2$. Using the same arguments as in the proofs of \cite[Lemmas 7.5 \& 7.6]{KL12}, 
we find for $\psi(x) \in (W \otimes F)^{\oplus2}$ the pointwise inequality
\begin{align*}
\big\la \lambda\til\pS(x) \psi(x) \bigmvert \lambda\til\pS(x) \psi(x) \big\ra + \big\la \big\{ \til\D , \lambda\til\pS(\cdot) \big\}(x) \psi(x) \bigmvert \psi(x) \big\ra 
\geq \frac12 \big\la \lambda\til\pS(x) \psi(x) \bigmvert \lambda\til\pS(x) \psi(x) \big\ra - \mu_x \la\psi(x)|\psi(x)\ra .
\end{align*}
Let $\epsilon > 0$, and pick a compactly supported smooth function $u\in C_c^\infty(M)$ such that $u(x)^2 \geq \epsilon + \mu_x$ for all $x\in K$ (note that $\sup_{x\in K}\mu_x$ is bounded by \cref{lem:comm_loc_est}, so such a function $u$ indeed exists). Then for $x\in K$ we can estimate
\begin{multline*}
\frac12 \big\la \lambda\til\pS(x) \psi(x) \bigmvert \lambda\til\pS(x) \psi(x) \big\ra - \mu_x \la\psi(x)|\psi(x)\ra + \la\psi(x)|u(x)^2\psi(x)\ra \\
\geq 0 + (u(x)^2-\mu_x) \la\psi(x)|\psi(x)\ra 
\geq \epsilon \la\psi(x)|\psi(x)\ra .
\end{multline*}
Next, let $c>0$ be such that $[-c,c]$ does not intersect the spectrum of $\pS(x)$ for any $x\in M\backslash K$, and write $\kappa := (1+c^{-1}) \left\| \big[\D,\pS(\cdot)\big] \big(\pS(\cdot)\pm i\big)^{-1}\right\|$. 
For $x\in M\backslash K$ we then know from \cref{lem:comm_loc_est} that $\mu_x\leq1+\kappa^2$, for \emph{any} choice of $\lambda$ (as long as $\lambda\geq1$). 
Now pick $\lambda_0 := \max\big(1,c^{-1} \sqrt{2(1+\kappa^2+\epsilon)}\big)$. Using the assumption that $\lambda\geq\lambda_0$, we obtain for $x\in M\backslash K$ the estimate
\begin{multline*}
\frac12 \big\la \lambda\til\pS(x) \psi(x) \bigmvert \lambda\til\pS(x) \psi(x) \big\ra - \mu_x \la\psi(x)|\psi(x)\ra + \la\psi(x)|u(x)^2\psi(x)\ra \\
\geq \frac12 \lambda_0^2 c^2 \la\psi(x)|\psi(x)\ra - \mu_x \la\psi(x)|\psi(x)\ra + 0 
\geq \epsilon \la\psi(x)|\psi(x)\ra .
\end{multline*}
Thus we have shown that for \emph{any} $x\in M$ we have the inequality
\[
\frac12 \big\la \lambda\til\pS(x) \psi(x) \bigmvert \lambda\til\pS(x) \psi(x) \big\ra - \mu_x \big\la \psi(x) \bigmvert \psi(x) \big\ra + \la\psi(x)|u(x)^2\psi\ra \geq \epsilon \la\psi(x)|\psi(x)\ra .
\]
For any $\psi\in\Dom\til\D_{\lambda\pS}^2$ we then find
\begin{align*}
\big\la \psi \bigmvert \big( \til\D_{\lambda\pS}^2 + u^2 \big) \psi \big\ra 
&= \big\la \til\D_{\lambda\pS} \psi \bigmvert \til\D_{\lambda\pS} \psi \big\ra + \la\psi|u^2\psi\ra \\
&= \big\la \til\D \psi \bigmvert \til\D \psi \big\ra + \big\la \lambda\til\pS(\cdot) \psi \bigmvert \lambda\til\pS(\cdot) \psi \big\ra + \big\la \til\D \psi \bigmvert \lambda\til\pS(\cdot) \psi \big\ra + \big\la \lambda\til\pS(\cdot) \psi \bigmvert \til\D \psi \big\ra + \la\psi|u^2\psi\ra \\
&\geq \big\la \lambda\til\pS(\cdot) \psi \bigmvert \lambda\til\pS(\cdot) \psi \big\ra + \big\la \big\{ \til\D , \lambda\til\pS(\cdot) \big\} \psi \bigmvert \psi \big\ra + \la\psi|u^2\psi\ra \\
&= \int_M \bigg( \big\la \lambda\til\pS(x) \psi(x) \bigmvert \lambda\til\pS(x) \psi(x) \big\ra + \big\la \big\{ \til\D , \lambda\til\pS(\cdot) \big\}(x) \psi(x) \bigmvert \psi(x) \big\ra \\
&\qquad\quad+ \la\psi(x)|u(x)^2\psi(x)\ra \bigg) \dvol(x) \\
&\geq \int_M \bigg( \frac12 \big\la \lambda\til\pS(x) \psi(x) \bigmvert \lambda\til\pS(x) \psi(x) \big\ra - \mu_x \la\psi(x)|\psi(x)\ra + \la\psi(x)|u(x)^2\psi(x)\ra \bigg) \dvol(x) \\
&\geq \epsilon \int_M \la\psi(x)|\psi(x)\ra \dvol(x) \\
&= \epsilon \la\psi|\psi\ra .
\end{align*}
Hence we have shown that the spectrum of $\til\D_{\lambda\pS}^2 + u^2$ is contained in $[\epsilon,\infty)$, and therefore we have a well-defined inverse $\big(\til\D_{\lambda\pS}^2 + u^2\big)^{-1} \in \End_B\big(L^2(M,E\otimes\bF)^{\oplus2}\big)$. The proof that $\til\D_{\lambda\pS}$ is Fredholm is then similar to the proof of \cref{thm:sum_Fredholm}. 
Pick a smooth function $\chi\in C_c^\infty(M)$ such that $0\leq\chi\leq1$, and $\chi(x)=1$ for all $x\in\supp u$. Write $\chi' := \sqrt{1-\chi^2}$. 
Using that $u\chi'=0$, we calculate that 
\begin{align*}
\til\D_{\lambda\pS} \chi' \til\D_{\lambda\pS} \big(\til\D_{\lambda\pS}^2+u^2\big)^{-1} \chi' 
&= [\til\D,\chi'] \til\D_{\lambda\pS} \big(\til\D_{\lambda\pS}^2+u^2\big)^{-1} \chi' + (\chi')^2 .
\end{align*}
As in the proof of \cref{thm:sum_Fredholm}, one can then check that we have a parametrix for $\til\D_{\lambda\pS}$ given by
\begin{align*}
Q &:= \chi \big(\til\D_{\lambda\pS}-i\big)^{-1} \chi + \chi' \til\D_{\lambda\pS} \big(\til\D_{\lambda\pS}^2+u^2\big)^{-1} \chi' . \qedhere
\end{align*}
\end{proof}

\begin{lem}
\label{lem:diff_homotopy}
Let $\D$ and $\pS(\cdot)$ be as in \cref{ass:diff}. 
Suppose that the graph norms of $\pS(x)$ are uniformly equivalent, and that $\pS\colon M\to\Hom_B(W,E)$ has a uniformly bounded weak derivative. 
Let $f\in C_0^1(M)$ be a differentiable function vanishing at infinity, such that $0<f(x)\leq1$ for all $x\in M$, $f(x)=1$ for all $x\in K$, and $\sup_{x\in M} \|df^{-1}(x)\| < \infty$. 
Consider the operator $\pS'(\cdot) := f^{-1}\pS(\cdot)$ corresponding to the family $\{f(x)^{-1}\pS(x)\}_{x\in M}$.
Then there exists $\lambda_0\geq1$ such that $\Index\big(\D-i\lambda\pS(\cdot)\big) = \Index\big(\D-i\pS'(\cdot)\big)$ for any $\lambda\geq\lambda_0$. 
\end{lem}
\begin{proof}
From \cref{lem:comm_loc_est} we know that $\big[\D,f^{-1}\pS(\cdot)\big] \big(f^{-1}\pS(\cdot)\pm i\big)^{-1}$ is well-defined and bounded. The same holds for the functions $f_t(x) := f(x)^t$ for $t\in[0,1]$. Therefore, considering the family of operators $\{\pS^\bullet(x)\}_{x\in M}$ on $C([0,1],E)$, given by $\pS^t(x) := f(x)^{-t}\pS(x)$, we find that $[\D,\pS^\bullet(\cdot)](\pS^\bullet(\cdot)\pm i)^{-1}$ is also well-defined and bounded. 
By \cref{lem:diff_Fred}, there exists a $\lambda_0\geq1$ such that $\D_{\lambda\pS^\bullet} = \D - i \lambda \pS^\bullet(\cdot)$ on $C([0,1],L^2(M,E\otimes\bF)^{\oplus2})$ is Fredholm for all $\lambda\geq\lambda_0$. This proves that $\Index\big(\D-i\lambda\pS(\cdot)\big) = \Index\big(\D-i\lambda\pS'(\cdot)\big)$. 
Finally, we know from \cref{lem:pS_cpt_res,prop:cpt_res} that $\til\D_{\lambda\pS'}$ has compact resolvents for any $\lambda>0$. Hence $\til\D_{\lambda\pS'}$ is homotopic to $\til\D_{\pS'}$, and $\Index\big(\D-i\lambda\pS(\cdot)\big) = \Index\big(\D-i\lambda\pS'(\cdot)\big) = \Index\big(\D-i\pS'(\cdot)\big)$. 
\end{proof}

\begin{prop}
\label{prop:diff_Kasp_prod}
Let $\D$ and $\pS(\cdot)$ be as in \cref{ass:diff}. 
Suppose that the graph norms of $\pS(x)$ are uniformly equivalent, and that $\pS\colon M\to\Hom_B(W,E)$ has a uniformly bounded weak derivative. 
Then there exists $\lambda_0\geq1$ such that for any $\lambda\geq\lambda_0$, the class $[\til\D_{\lambda\pS}] \in \KK^0(\C,B)$ is the Kasparov product of $[\pS(\cdot)] \in \KK^1(\C,C_0(M,B))$ with $[\D] \in \KK^1(C_0(M),\C)$. 
\end{prop}
\begin{proof}
Let $f\in C_0^1(M)$ be a differentiable function vanishing at infinity, such that $0<f(x)\leq1$ for all $x\in M$, $f(x)=1$ for all $x\in K$, and $\sup_{x\in M} \|df^{-1}(x)\| < \infty$ (as in \cite[Lemma 8.10]{KL13}, such functions exist). 
By \cref{lem:pS_cpt_res}, the operator $\pS'(\cdot) := f^{-1}\pS(\cdot)$ corresponding to the family $\{f(x)^{-1}\pS(x)\}_{x\in M}$ defines an odd unbounded Kasparov module $(\C,C_0(M,E)_{C_0(M,B)},\pS'(\cdot))$, and we have $[\pS'(\cdot)] = [\pS(\cdot)]$. From \cref{lem:diff_homotopy} we also know that $[\til\D_\pS'] = [\til\D_{\lambda\pS}]$. Since $f^{-1}\pS(\cdot)$ has compact resolvents, 
we know from \cref{prop:cpt_res} that $\til\D_\pS'$ has compact resolvents, and therefore $\til\D_\pS'$ defines an (even) unbounded Kasparov $\C$-$\C$-module $(\C,L^2(M,E\otimes\bF)^{\oplus2},\til\D_\pS')$. 

Thus it remains to be proven that $[\til\D_\pS']$ is the Kasparov product of $[\pS'(\cdot)]$ with $[\D]$, and for this purpose we need to check the three conditions in Kucerovsky's \cref{thm:Kucerovsky}. 

Let $\psi,\eta\in C_c^\infty(M,W)$ and $\xi_1,\xi_2\in\Dom\D$, and let $\sigma$ denote the principal symbol of $\D$. Then 
\begin{align*}
\left[ \mattwo{\D\pm i\pS'(\cdot)}{0}{0}{\D} , \mattwo{0}{T_\psi}{T_\psi^*}{0} \right] \vectwo{\eta\otimes\xi_1}{\xi_2} 
&= \vectwo{(\D\pm i\pS'(\cdot)) (\psi\otimes\xi_2) - \psi\otimes\D\xi_2}{\D \la\psi|\eta\ra \xi_1 - \la\psi|\pm i\pS'(\cdot)\eta\ra \xi_1 - \la\psi|\sigma(d\eta)\ra \xi_1 - \la\psi|\eta\ra \D\xi_1 } \\
&= \vectwo{\pm iT_{\pS'(\cdot)\psi} \xi_2 + T_{\sigma(d\psi)} \xi_2}{\mp iT_{\pS'(\cdot)\psi}^*(\eta\otimes\xi_1) + T^*_{\sigma(d\psi)}(\eta\otimes\xi_1) } \\
&= \mattwo{0}{\pm iT_{\pS'(\cdot)\psi} + T_{\sigma(d\psi)}}{\mp iT_{\pS'(\cdot)\psi}^* + T^*_{\sigma(d\psi)}}{0} \vectwo{\eta\otimes\xi_1}{\xi_2} .
\end{align*}
Here we have used the notation 
$T_{\sigma(\phi\otimes\alpha)} = T_\phi \sigma(\alpha)$, 
$T^*_{\sigma(\phi\otimes\alpha)} = T^*_\phi \sigma(\alpha)$, and 
$\la\phi|\sigma(\eta\otimes\alpha)\ra = \la\phi|\eta\ra \sigma(\alpha)$, 
for any $\phi,\eta\in C_c^\infty(M,W)$ and $\alpha\in \Gamma_c^\infty(T^*M)$. 
For each $\psi\in C_c^\infty(M,W)$, the operators $T_{\pS'(\cdot)\psi}$, $T_{\sigma(d\psi)}$, and $T^*_{\sigma(d\psi)}$ are bounded. Hence the above commutator is bounded on $C_c^\infty(M,W) \otimes \Dom\D$. Since $C_c^\infty(M,W) \otimes \Dom\D$ is a core for $\D\pm i\pS'(\cdot)$, we conclude that for each $\psi\in C_c^\infty(M,W)$ (which is dense in $C_0(M,E)$) this commutator is bounded on $\Dom(\D\pm i\pS'(\cdot)) \oplus \Dom\D$, which shows the first condition.

The second condition requires the domain inclusion $\Dom\til\D_\pS'\subset\Dom\til\D$. 
From \cref{lem:comm_loc_est} we know that $\big[\D,\pS'(\cdot)\big] \big(\pS'(\cdot)\pm i\big)^{-1}$ is well-defined and bounded. It then follows from \cite[Theorem 7.10]{KL12} that $\til\D_\pS'$ is regular self-adjoint on the domain $\Dom\til\pS'(\cdot)\cap\Dom\til\D$. In particular, the second condition holds. 

Finally, using again the boundedness of $[\D,\pS'(\cdot)](\pS'(\cdot)\pm i)^{-1}$, we know from \cite[Lemma 7.6]{KL12} that there exists a constant $c>0$ such that 
\[
\big\la (\D\pm i\pS'(\cdot))\psi | (\D\pm i\pS'(\cdot))\psi \big\ra \geq \frac12 \big\la \pS'(\cdot)\psi | \pS'(\cdot)\psi \big\ra + \la\D\psi|\D\psi\ra - c \la\psi|\psi\ra . 
\]
for all $\psi\in\Dom(\D\pm i\pS'(\cdot))$. 
This implies that
\[
2 \la\pS'(\cdot)\psi|\pS'(\cdot)\psi\ra \mp i \la\pS'(\cdot)\psi|\D\psi\ra \pm i \la\D\psi|\pS'(\cdot)\psi\ra 
\geq -c \la\psi|\psi\ra ,
\]
which shows the third condition. 
\end{proof}

\subsection{A continuous family}

We now consider the general case without assuming any differentiability for the family $\{\pS(x)\}_{x\in M}$. 
We will show that $\pS(\cdot)$ is nevertheless homotopic to an operator obtained from a differentiable family. 
For this purpose, 
we cite (a special case of) Wockel's generalisation of Steenrod's approximation theorem.

\begin{thm}[Generalised Steenrod Approximation Theorem \cite{Woc09}]
\label{thm:Wockel-Steenrod}
Let $M$ be a finite-dimensional connected manifold, $\pi\colon E\to M$ be a locally trivial smooth bundle with a locally convex manifold $N$ as typical fibre, and $\sigma\colon M\to E$ be a continuous section. 
If $C\subset M$ is closed such that $\sigma$ is smooth on a neighbourhood of $C$, then for each open neighbourhood $O$ of $\sigma(M)$ in $E$, there exists a section $\tau\colon M\to O$ which is smooth on $M$ and equals $\sigma$ on $C$. 
Furthermore, there exists a homotopy $F\colon[0,1]\times M\to O$ between $\sigma$ and $\tau$ such that each $F(t,\cdot)$ is a section of $\pi$ and $F(t,x) = \sigma(x) = \tau(x)$ if $(t,x)\in[0,1]\times C$. 
\end{thm}

From here on, we consider $M$, $\D$, and $\pS(\cdot)$ satisfying the \ref{ass:standing} as well as assumption \ref{ass:A4'}. 

\begin{lem}
\label{lem:pS_diff}
There exists a family $\{H^t(x)\}_{(t,x)\in[0,1]\times M}$, 
such that $H^0(x) = \pS(x)$ for all $x\in M$, 
such that $H^1(\cdot)\colon M\to\Hom_B(W,E)$ has a uniformly bounded weak derivative, 
and such that the family of operators $\{H^\bullet(x)\}_{x\in M}$ on the Hilbert $C([0,1],B)$-module $C([0,1],E)$ satisfies the assumptions \ref{ass:A1}-\ref{ass:A4'}. 
\end{lem}
\begin{proof}
Let $\Hom_B^s(W,E)$ denote the (real) Banach space of symmetric operators $S$ on $E$ with $\Dom S = W$ (equipped with the operator norm of maps $W\to E$). 
We consider the globally trivial bundle $M\times \Hom_B^s(W,E)$. Then $\pS$ is a continuous section of $M\times \Hom_B^s(W,E)$. 
For any $x\in M$, consider the open neighbourhood $O_x$ of $\pS(x)$ in $\Hom_B^s(W,E)$ given by
\[
O_x := \left\{ T\in \Hom_B^s(W,E) : \|(T-\pS(x)) (\pS(x)\pm i)^{-1}\| < \frac12 \right\} .
\]
Let $O$ be any open neighbourhood of 
$\pS(M) \subset M\times \Hom_B^s(W,E)$ 
which is contained in $\bigsqcup_{x\in M} O_x$. 

Using \cref{prop:pS_cst_ends}, we may assume (without loss of generality) that $\pS(x) = \pS(x_j)$ for all $x\in V_j$. 
Let $U$ be a precompact open neighbourhood of $K$. 
Applying \cref{thm:Wockel-Steenrod}, there exists a homotopy $H\colon[0,1]\times M\to O$ such that $H^0(\cdot) = \pS(\cdot)$, $\pS'(\cdot) := H^1(\cdot)$ is smooth, and $H^t(x) = \pS(x) = \pS'(x)$ for all $x\in M\backslash U$. 
Since $H^t(x)$ is a symmetric perturbation of $\pS(x)$ such that $\|(H^t(x)-\pS(x)) (\pS(x)\pm i)^{-1}\| < \frac12$, 
it follows from the Kato-Rellich theorem that $H^t(x)$ is self-adjoint with $\Dom H^t(x) = \Dom\pS(x) = W$ for any $x\in M$. 
Furthermore, the family $\{H^\bullet(x)\}_{x\in M}$ again satisfies the assumptions \ref{ass:A1}-\ref{ass:A4'}. 
By compactness of $\bar U$, and the fact that $\pS'(x)$ is locally constant outside of $\bar U$, we know that the derivative of $\pS'$ is uniformly bounded on all of $M$. 
\end{proof}

\begin{prop}
\label{prop:Kasp_prod_complete}
Suppose that $M$ is complete, and that $\D$ has bounded propagation speed. 
Then the class $[\til\D_\pS] \in \KK^0(\C,B)$ is the (internal) Kasparov product of $[\pS(\cdot)] \in \KK^1(\C,C_0(M,B))$ with $[\D] \in \KK^1(C_0(M),\C)$. 
\end{prop}
\begin{proof}
Let $\lambda>0$, and note that $\lambda\pS(\cdot)$ still satisfies the assumptions \ref{ass:A1}-\ref{ass:A4'}. 
By \cref{lem:pS_diff}, we have a family $\{H^t(x)\}_{(t,x)\in[0,1]\times M}$, such that $H^0(x) = \lambda\pS(x)$ for all $x\in M$, such that 
$\lambda\pS'(\cdot) := H^1(\cdot)\colon M\to\Hom_B(W,E)$ has a uniformly bounded weak derivative, 
and such that the family of operators $\{H^\bullet(x)\}_{x\in M}$ on the Hilbert $C([0,1],B)$-module $\til E := C([0,1],E)$ satisfies the assumptions \ref{ass:A1}-\ref{ass:A4'}. Let $\til\D_{H^\bullet} = H^\bullet(\cdot)\times\D$ be the corresponding product operator. 
By \cref{thm:sum_Fredholm}, $\til\D_{H^\bullet}$ is regular self-adjoint and Fredholm 
on $L^2(M,\til E\otimes\bF)^{\oplus2} \simeq C([0,1],L^2(M,E\otimes\bF)^{\oplus2})$. 
Hence $\til\D_{H^\bullet}$ is a homotopy between $\til\D_{\lambda\pS'} = \til\D_{H^1}$ and $\til\D_{\lambda\pS} = \til\D_{H^0}$, and therefore $[\til\D_{\lambda\pS'}] = [\til\D_{\lambda\pS}]$. 
By \cref{prop:diff_Kasp_prod}, there exists a $\lambda\geq1$ such that 
\begin{align*}
[\pS(\cdot)] \otimes_{C_0(M)} [\D] &= [\pS'(\cdot)] \otimes_{C_0(M)} [\D] = [\til\D_{\lambda\pS'}] = [\til\D_{\lambda\pS}] .
\end{align*}
Since $[\til\D_{\lambda\pS}] = [\til\D_{\pS}]$ by \cref{lem:Fredholm_rescaling}, this completes the proof. 
\end{proof}

The following useful consequence of the relative index theorem allows us to replace the manifold $M$ by a manifold with cylindrical ends. 

\begin{prop}
\label{prop:index_cut-and-paste}
There exist a precompact open subset $U$ of $M$ and a Dirac-Schr\"odinger operator $\D_\pS'$ on the manifold $M' := \bar U \cup_{\partial U} \big(\partial U\times[0,\infty)\big)$ satisfying the \ref{ass:standing} and assumption \ref{ass:A4'}, such that 
\begin{enumerate}
\item the operators $\D'$ and $\pS'(\cdot)$ on $M'$ agree with $\D$ and $\pS(\cdot)$ on $M$ when restricted to $U$;
\item the metric and the operators $\D'$ and $\pS'(\cdot)$ on $M'$ are of product form on $\partial U\times[1,\infty)$;
\item we have the equality $[\til\D_\pS'] = [\til\D_\pS] \in \KK^0(\C,B)$. 
\end{enumerate}
\end{prop}
\begin{proof}
Let $U$ be a precompact open neighbourhood of the interior of $K$, with smooth compact boundary $\partial U$. 
We may choose $U$ such that $x_j\in\partial U$ (where $x_j$ are the points given in assumption \ref{ass:A4'}). 
By assumption, $\partial U$ is the finite disjoint union of $N_j := V_j\cap \partial U$. 
Consider the manifold $M' := \bar U \cup_{\partial U} \big(\partial U\times[0,\infty)\big)$ with cylindrical ends. 
Let $C\simeq\partial U\times(-\frac12,\frac12)$ be a tubular neighbourhood of $\partial U$, such that there exists an isometry $\phi\colon U\cup C \to \bar U\cup_{\partial U}\big(\partial U\times[0,\frac12)\big) \subset M'$ (which preserves the subset $U$). 
Equip $M'$ with a Riemannian metric which is of product form on $\partial U\times[1,\infty)$, and which agrees with $g|_U$ on $U$. Let $\bF'\to M'$ be a hermitian vector bundle which agrees with $\bF|_U$ on $U$. Let $\D'$ be a symmetric elliptic first-order differential operator on $\bF'\to M'$, which is of product form on $\partial U\times[1,\infty)$, and which agrees with $\D|_{U\cup C}$ on $U\cup C$. Consider the family $\{\pS'(x)\}_{x\in M'}$ given by
\[
\pS'(x) := \begin{cases}
            \pS(x) , & x\in U , \\
            \pS(y) , & x=(y,r)\in \partial U\times[0,\infty) . 
            \end{cases}
\]
Then the family $\{\pS'(x)\}_{x\in M'}$ also satisfies the assumptions \ref{ass:A1}-\ref{ass:A4'}, where the open subsets $V_j' = N_j\times(0,\infty)$ satisfy assumption \ref{ass:A4'} (with the same points $x_j\in N_j := V_j\cap\partial U$). 
Thus we have constructed a Dirac-Schr\"odinger operator $\til\D_\pS' := \pS'(\cdot) \times \D'$ on $M'$, satisfying the desired properties 1) and 2). 
It remains to prove the equality $[\til\D_\pS'] = [\til\D_\pS]$, for which we invoke the relative index theorem. 

Let $M^1 := M$ and $M^2 := \partial U \times \R$ (i.e., $M^2$ is the finite disjoint union of the cylindrical manifolds $N_j\times\R$). 
Let $C' = \phi(C)$ be the collar neighbourhood of $\partial U$ in $M'$. 
We equip $M^2$ with a complete Riemannian metric which agrees with the metric of $M'$ on $C' \cup \big(\partial U\times(0,\infty)\big)$, and which is of product form on $(-\infty,-1]\times\partial U$. We extend the vector bundle $\bF'|_{C'\cup (\partial U\times(0,\infty))}$ to a bundle $\bF^2\to M^2$, and we pick an operator $\D^2$ (satisfying the \ref{ass:standing}) such that $\D^2|_{C'\cup(\partial U\times(0,\infty))} = \D'|_{C'\cup(\partial U\times(0,\infty))}$ (for instance, we can take $\D^2$ to be of product form on $(-\infty,-1]\times\partial U$). 
We define a family $\{\pS^2(x)\}_{x\in M^2}$ by $\pS^2(y,r) := \pS(y)$ for all $y\in\partial U$ and $r\in\R$. 
Then $\bF^2\to M^2$, $\D^2$, and $\pS^2(\cdot)$ satisfy the \ref{ass:standing}. By cutting and pasting along $\partial U$, we obtain manifolds $M^3 = M'$ and $M^4 = \big(\partial U\times(-\infty,0]\big) \cup_{\partial U} (M\backslash U)$, with corresponding operators $\D^3$, $\pS^3(\cdot)$, $\D^4$, and $\pS^4(\cdot)$. By \cref{thm:rel_index} we have $[\til\D_\pS^1] + [\til\D_\pS^2] = [\til\D_\pS^3] + [\til\D_\pS^4] \in \KK^0(\C,B)$. 
The manifolds $M^2$ and $M^4$ are both given by a finite disjoint union of manifolds satisfying the assumptions of \cref{lem:inv_glob}, and therefore $[\til\D_\pS^2] = [\til\D_\pS^4] = 0$. 
Since $M^1 = M$ and $M^3 = M'$, we conclude that $[\til\D_\pS] = [\til\D_\pS']$. 
\end{proof}

\begin{thm}
\label{thm:Kasp_prod_index}
Let $M$ be any connected Riemannian manifold, and consider operators $\pS(\cdot)$ and $\D$ satisfying the \ref{ass:standing} and assumption \ref{ass:A4'}. 
Then $[\til\D_\pS] \in \KK^0(\C,B)$ is the (internal) Kasparov product (over $C_0(M)$) of $[\pS(\cdot)] \in \KK^1(\C,C_0(M,B))$ with $[\D] \in \KK^1(C_0(M),\C)$. 
\end{thm}
\begin{proof}
From \cref{prop:index_cut-and-paste}, 
we obtain a Dirac-Schr\"odinger operator $\til\D_\pS'$ on a complete manifold $M' = \bar U \cup_{\partial U} \big(\partial U\times[0,\infty)\big)$ with (finitely many) cylindrical ends $N_j\times[0,\infty)$ satisfying the \ref{ass:standing} and \ref{ass:A4'}, such that $\D'$ has bounded propagation speed
and we have the equality $[\til\D_\pS'] = [\til\D_\pS]$. 
In particular, the product operator $\til\D_\pS' := \pS'(\cdot) \times \D'$ on $M'$ satisfies the assumptions of \cref{prop:Kasp_prod_complete}. 

Recall from \cref{sec:symm_ell} that, for any open subset $U\subset M$, the restriction $\D|_U$ of $\D$ to $L^2(U,\bF|_U)$ defines a class $[\D|_U] \in \KK^1(C_0(U),\C)$. 
From \cref{lem:pS_res,prop:D_2_res}, we have the equalities
\[
[\pS(\cdot)] \otimes_{C_0(M)} [\D] = [\pS(\cdot)|_U] \otimes_{C_0(U)} [\iota_U] \otimes_{C_0(M)} [\D] = [\pS(\cdot)|_U] \otimes_{C_0(U)} [\D|_U] .
\]
Since we have the same equalities on $M'$, we find that
\begin{align*}
[\pS(\cdot)] \otimes_{C_0(M)} [\D] &= [\pS(\cdot)|_U] \otimes_{C_0(U)} [\D|_U] = [\pS'(\cdot)] \otimes_{C_0(M')} [\D'] .
\end{align*}
Hence from \cref{prop:Kasp_prod_complete} we conclude that 
\begin{align*}
[\pS(\cdot)] \otimes_{C_0(M)} [\D] &= [\pS'(\cdot)] \otimes_{C_0(M')} [\D'] = [\til\D_\pS'] = [\til\D_\pS] .
\qedhere
\end{align*}
\end{proof}

Combining \cref{thm:Kasp_prod_index} with \cref{prop:SF_Kasp_prod}, we obtain as a corollary the following `index = spectral flow' equality, generalising 
\cref{thm:index=sf} to the case of Hilbert modules. 

\begin{coro}[`Index = spectral flow']
\label{coro:ind_sf}
Let $\{\pS(x)\}_{x\in \R}$ be a family of regular self-adjoint operators on a Hilbert $B$-module $E$, satisfying the \ref{ass:standing} (with $M=\R$ and $\D=-i\partial_x$), and satisfying \ref{ass:A4'} with the open subsets $V_1 = (-\infty,a)$ and $V_2 = (b,\infty)$ for some $a<b\in\R$. 
Suppose furthermore that there exist locally trivialising families for $\{\pS(x)\}_{x\in\R}$. 
Then we have the equality 
\[
\Index \big( \partial_x + \pS(\cdot) \big) = \SF\big(\{\pS(x)\}_{x\in\R}\big) \in \K_0(B) .
\]
\end{coro}

\begin{example}
Consider a cylindrical manifold $M=\R\times N$, where $N$ is a smooth compact manifold. 
We assume that $M$ is equipped with the product metric obtained from the standard metric on $\R$ and some fixed metric on $N$ (in particular, $M$ is complete). 
Let $\D_N$ be a symmetric elliptic first-order differential operator on $N$, which yields a $\K$-homology class $[\D_N] \in \KK^1(C(N),\C)$. 
Consider the standard spectral triple $(C_c^\infty(\R) , L^2(\R) , -i\partial_r )$ on the real line, which yields a class $[-i\partial_r] \in \KK^1(C_0(\R),\C)$. 
From $-i\partial_r$ and $\D_N$ we construct the (essentially) self-adjoint elliptic first-order differential operator $\D$ which is of product form on $M$. This implies that $[\D] = [-i\partial_r] \otimes [\D_N] \in \KK^1(C_0(M),\C)$. 

Consider a family of self-adjoint operators $\{\pS(r,y)\}_{(r,y)\in\R\times N}$ on a Hilbert space $\mH$, satisfying assumptions \ref{ass:A1}-\ref{ass:A4'}. For simplicity, let us assume that $\pS(r,y) = \pS(0,y)$ for all $r\leq0$, and $\pS(r,y) = \pS(1,y)$ for all $r\geq1$. 
We then obtain a regular self-adjoint Fredholm operator $\pS(\cdot,\cdot)$ on $C_0(M,\mH)$, which yields a class $[\pS(\cdot,\cdot)] \in \KK^1(\C,C_0(M)) = \KK^1(\C,C_0(\R,C(N)))$. 
As in \cref{eg:sf_cyl}, writing $\pS^y(r) := \pS(r,y)$, we obtain a family of regular self-adjoint operators $\{\pS^\bullet(r)\}_{r\in\R}$ on the Hilbert $C(N)$-module $C(N,\mH)$. The class $[\pS^\bullet(\cdot)] \in \KK^1(\C,C_0(\R,C(N)))$ then corresponds to the spectral flow of $\{\pS^\bullet(r)\}_{r\in[0,1]}$. 

For $C^*$-algebras $A$, $B$, and $C$, recall the map $\tau_C\colon \KK(A,B)\to\KK(A\otimes C,B\otimes C)$ given by the external Kasparov product with the identity element $1_C\in\KK(C,C)$. 
Using the properties of the Kasparov product, we then obtain
\begin{align*}
\Index \D_\pS &= [\pS(\cdot,\cdot)] \otimes_{C_0(M)} [\D] 
= [\pS(\cdot,\cdot)] \otimes_{C_0(M)} \big( [-i\partial_r] \otimes [\D] \big) 
= [\pS^\bullet(\cdot)] \otimes_{C_0(M)} \tau_{C(N)}( [-i\partial_r] ) \otimes_{C(N)} [\D_N] \\
&= \big( [\pS^\bullet(\cdot)] \otimes_{C_0(\R)} [-i\partial_r] \big) \otimes_{C(N)} [\D_N] 
= \beta\big(\ASF(\pS^\bullet(\cdot))\big) \otimes_{C(N)} [\D_N] ,
\end{align*}
where on the last line we used \cref{prop:SF_Kasp_prod}. 
Assuming there exist locally trivialising families for $\{\pS^\bullet(r)\}_{r\in[0,1]}$, we obtain from \cref{prop:ASF_R} the equality 
\[
\Index \D_\pS = \SF\big(\{\pS^\bullet(r)\}_{r\in[0,1]}\big) \otimes_{C(N)} [\D_N] .
\]
\end{example}

\subsection{Incomplete manifolds and manifolds with boundary}
\label{sec:index_non-sa}

From \cref{prop:index_cut-and-paste} we know that we can cut and paste certain open ends of the manifold $M$ without changing the index of $\D-i\pS(\cdot)$. 
More precisely, suppose that $\pS(\cdot)$ is uniformly invertible outside a compact set $K$, and let $U$ be a precompact open neighbourhood of $K$ with smooth boundary $\partial U$, such that $M\backslash\bar U$ is a finite disjoint union of open subsets on which the variation of $\pS(\cdot)$ is `sufficiently small' (as in assumption \ref{ass:A4'}). 
Then in particular the index of $\D-i\pS(\cdot)$ is completely determined by its restriction $\D-i\pS(\cdot)|_U$ to the subset $U$, even though $\D-i\pS(\cdot)|_U$ may not be Fredholm. This observation allows us to obtain a well-defined index also more generally, starting with a possibly \emph{non-self-adjoint} operator $\D$. 

Consider the following setup, in which we drop the self-adjointness of $\D$ from the \ref{ass:standing}. 
Let $(M,g)$ be a Riemannian manifold. Let $\D$ be a symmetric elliptic first-order differential operator on a hermitian vector bundle $\bF\to M$. Let $\{\pS(x)\}_{x\in M}$ be a family of regular self-adjoint operators on a Hilbert $B$-module $E$ satisfying the assumptions \ref{ass:A1}-\ref{ass:A3}. We fix a compact subset $K$ of $M$ such that $\pS(\cdot)$ is uniformly invertible outside of $K$, and a precompact open neighbourhood $U$ of $K$ with smooth compact boundary $\partial U = \bigsqcup_j N_j$ which consists of finitely many connected components $\{N_j\}$. 
Assume that there exist points $x_j\in N_j$ and positive numbers $a_j<1$ such that $\big\| \big( \pS(x)-\pS(x_j) \big) S(x_j)^{-1} \big\| \leq a_j$ for all $x\in N_j$. 

Consider the manifold $M' := \bar U \cup_{\partial U} (\partial U\times[0,\infty))$. Equip $M'$ with a Riemannian metric which is of product form on $\partial U\times[1,\infty)$, and which agrees with $g|_U$ on $U$. Let $\bF'\to M'$ be a hermitian vector bundle which agrees with $\bF|_U$ on $U$. Let $\D'$ be a symmetric elliptic first-order differential operator of bounded propagation speed on $\bF'\to M'$, which agrees with $\D|_U$ on $U$. Consider the family $\{\pS'(x)\}_{x\in M'}$ given by
\[
\pS'(x) := \begin{cases}
            \pS(x) , & x\in U , \\
            \pS(y) , & x=(y,r)\in \partial U\times[0,\infty) . 
            \end{cases}
\]
Then the family $\{\pS'(x)\}_{x\in M'}$ satisfies assumptions \ref{ass:A1}-\ref{ass:A4'} (for assumption \ref{ass:A4'}, we can take the open subsets $V_j' := N_j\times(0,\infty)$). 
Hence we obtain a regular self-adjoint Fredholm operator $\til\D_\pS' := \pS'(\cdot)\times\D'$, and therefore we have a well-defined class
\[
[\til\D_\pS]^\cyl := \Index\big(\D'-i\pS'(\cdot)\big) \in \K_0(B) .
\]
From \cref{lem:pS_res,prop:D_2_res,thm:Kasp_prod_index}, we obtain the equality $[\til\D_\pS]^\cyl = [\pS(\cdot)] \otimes_{C_0(M)} [\D]$. The main difference is that $[\til\D_\pS]^\cyl$ can no longer be expressed as the index of $\D-i\pS(\cdot)$. 

A similar procedure as above also allows us to deal with a Dirac-Schr\"odinger operator on a compact manifold with boundary. 
Indeed, viewing $\bar U$ as a compact manifold with boundary $\partial U$ and interior $U$, we obtain a class $[\til\D_\pS]^\cyl \in \K_0(B)$ by attaching a cylindrical end (as described above). 
In this case, it would be interesting to see if this class $[\til\D_\pS]^\cyl$ is equal to the index of a Fredholm operator on $\bar U$, obtained as an extension of $(\D-i\pS(\cdot))|_U$ by imposing suitable boundary conditions on $\partial U$. 
In the case of a \emph{classical} Dirac-Schr\"odinger operator (in which case the potential acts on a \emph{finite-rank} bundle), it has been shown in \cite{Rad94} that this is indeed the case (using a natural choice of boundary conditions on $\bar M$). It is left as an open problem whether or not this result can be generalised to our setup.


\providecommand{\noopsort}[1]{}\providecommand{\vannoopsort}[1]{}
\providecommand{\bysame}{\leavevmode\hbox to3em{\hrulefill}\thinspace}
\providecommand{\MR}{\relax\ifhmode\unskip\space\fi MR }
\providecommand{\MRhref}[2]{%
  \href{http://www.ams.org/mathscinet-getitem?mr=#1}{#2}
}
\providecommand{\href}[2]{#2}

\end{document}